\documentclass{amsart}
\usepackage{amsbsy}
\usepackage{amsmath}
\usepackage{amsthm}
\usepackage{amssymb}
\usepackage{color}

\makeatletter
\def\cleardoublepage{\clearpage\if@twoside \ifodd\c@page\else%
         \hbox{}%
     \thispagestyle{empty}
     \newpage%
     \if@twocolumn\hbox{}\newpage\fi\fi\fi}
\makeatother
\let\cleardoublepage\clearpage
\def\eq#1{(\ref{#1})}

\newtheorem{thm}{Theorem}[section]

\newtheorem{lem}[thm]{Lemma}
\newtheorem{pro}[thm]{Proposition}
\newtheorem{den}[thm]{Definition}

\numberwithin{equation}{section}
\newcommand{\eps}{\varepsilon}

\def\ds{\displaystyle}
\newcommand{\R}{\mathbb R}
\newcommand{\N}{\mathbb{N}}

\begin{document}

\title[Stability and qualitative properties of radial solutions]
{Stability and qualitative properties of radial solutions of the
Lane-Emden-Fowler equation on Riemannian models}

\author[Elvise Berchio]{Elvise Berchio}
\address{\hbox{\parbox{5.7in}{\medskip\noindent{Politecnico di Milano,\\
        Dipartimento di Ma\-t\-ema\-ti\-ca ``Francesco Brioschi'', \\
        Piazza Leonardo da Vinci 32, 20133 Milano, Italy. \\[3pt]
        \em{E-mail address: }{\tt elvise.berchio@polimi.it}}}}}
\author[Alberto Ferrero]{Alberto Ferrero}
\address{\hbox{\parbox{5.7in}{\medskip\noindent{Universit\`a degli Studi del Piemonte Orientale,\\
Dipartimento di Scienze e Innovazione Tecnologica, \\
        Viale Teresa Michel 11, 15121 Alessandria, Italy. \\[3pt]
        \em{E-mail addresses: }{\tt
          alberto.ferrero@mfn.unipmn.it}}}}}
\author[Gabriele Grillo]{Gabriele Grillo}
\address{\hbox{\parbox{5.7in}{\medskip\noindent{Politecnico di Milano,\\
        Dipartimento di Ma\-t\-ema\-ti\-ca ``Francesco Brioschi'', \\
        Piazza Leonardo da Vinci 32, 20133 Milano, Italy. \\[3pt]
        \em{E-mail address: }{\tt gabriele.grillo@polimi.it}}}}}




\thanks{E. Berchio and
A. Ferrero were partially supported by the PRIN 2008 grant
  ``Aspetti geometrici delle equazioni alle derivate parziali e
  questioni connesse''. G. Grillo was partially supported by the PRIN 2009 grant ``Metodi di viscosit\`a,
  geometrici e di controllo per modelli diffusivi nonlineari".}

\begin{abstract} We study existence, uniqueness and stability
of radial solutions of the Lane-Emden-Fowler equation $-\Delta_g
u=|u|^{p-1}u$ in a class of Riemannian models $(M,g)$ of dimension
$n\ge 3$ which includes the classical hyperbolic space $\mathbb
H^n$ as well as manifolds with sectional curvatures unbounded
below. Sign properties and asymptotic behavior of solutions are
influenced by the critical Sobolev exponent while the so-called
Joseph-Lundgren exponent is involved in the stability of solutions.
\end{abstract}

\maketitle

\section{Introduction}
\noindent We study the Lane-Emden-Fowler equation
\begin{equation}\label{LEF}
-\Delta_g u=|u|^{p-1}u\ \ \ {\rm on}\ M,
\end{equation}
where $n\ge 3$ and $p>1$, posed on a \it Riemannian model $(M,g)$ \rm, namely on a manifold admitting a pole $o$
and whose metric is given, in polar or spherical coordinates around $o$, by
\begin{equation}\label{metric}
ds^2=dr^2+(\psi(r))^2d\Theta^2,\ \ r>0, \Theta\in {\mathbb S}^{n-1}
\end{equation}
for a given function $\psi$ satisfying appropriate conditions. We will denote by $g$ this metric.
Here $d\Theta^2$ denotes the canonical metric on the unit
sphere ${\mathbb S}^{n-1}$, $r$ is by construction the Riemannian
distance between a point whose coordinates are $(r,\Theta)$ and
$o$, the function $\psi$ is smooth and positive on $(0,R)$ for
some $R\in(0,+\infty]$. In principle $R$ can be finite and in such a case it identifies the \it cut
locus \rm of $o$ in $M$, but hereafter and without further comments we shall assume that $R=+\infty$.

The additional assumptions we shall make later on $\psi$ correspond
to considering manifolds which have infinite volume and, at least
outside a compact set, have \it strictly negative \rm sectional
curvatures. Hence, if such condition holds globally, we are
dealing with special classes of Cartan-Hadamard manifolds. The
motivating example we have in mind is, therefore, the hyperbolic
space ${\mathbb H}^n$, in which some of the problems that we shall
study here in greater generality have been recently investigated.
In fact, the Riemannian model associated to the choice
$\psi(r)=\sinh r$ in \eqref{metric} is a well-known representation
of ${\mathbb H}^n$.

In the seminal paper \cite{mancini} among other results it is
shown that, for $p\in \left(1,\frac{n+2}{n-2}\right)$, there is a
unique strictly positive radial solution $U$ of \eqref{LEF}
belonging to the Sobolev space $H^1({\mathbb H}^n):=\{u\in
L^2(V_g);\, \nabla_g u\in L^2(V_g)\}$, where $V_g$ is
the Riemannian measure and $\nabla_g$ the Riemannian gradient, and $U$ is radial in the sense that it depends only on $r$.
This is in sharp contrast with the Euclidean case, where no such
solution exists, and is strongly related to the fact that the
$L^2$ spectrum of $-\Delta_g$ is bounded away from zero, so that
an $L^2$-Poincar\'e inequality holds. The solution $U$ is rapidly
decaying at infinity, but infinitely many other radial positive
solutions exist. The precise asymptotics of such slowly decaying
solutions was given in \cite{BGGV} also for the case $p\ge
\frac{n+2}{n-2}$, together with a classification of radial solutions in terms of their sign properties, further investigated in \cite{BK}. In fact, sign changing solutions may also exist and are
studied in \cite{BS, BGGV}: they can have finite or infinite
$H^1_r$ norm, and their asymptotics depend on which of the two
cases holds. In \cite{GS}, the critical case $p=\frac{n+2}{n-2}$
is investigated in further details. See also \cite{cfms} for other
results concerning elliptic problems and \cite{BPT} for semilinear parabolic problems in ${\mathbb H}^n$.

Our results aim at discussing the cases corresponding to the
defining function $\psi$ being everywhere increasing and,
moreover, such that $l:=\liminf_{r\rightarrow
+\infty}\frac{\psi'(r)}{\psi(r)}>0$ (which is Assumption $(H_3)$
in Section \ref{s:main}). While clearly the hyperbolic space
satisfies such condition, Riemannian models which are
asymptotically hyperbolic satisfy it as well and, more
importantly, such a condition allows for \it unbounded \rm
negative sectional curvatures: a typical example in which this can
hold corresponds to the choice $\psi(r)=e^{r^a}$ for a given $a>1$
and $r$ large, a case for which (see Section \ref{notation}) sectional curvatures in the radial direction diverge as $-a^2r^{2(a-1)}$ as $r\to+\infty$. In addition it will be shown later that, under the
stated assumption, the $L^2$ spectrum of $-\Delta_g$ is still
bounded away from zero, whereas if $\lim_{r\rightarrow
+\infty}\frac{\psi'(r)}{\psi(r)}=0$ then there is no gap in the
$L^2$ spectrum of $-\Delta_g$. Hence, one hardly expects in such
situation to be able to construct a positive solution to the
equation at hand. It is worth noticing here that if the radial
sectional curvature goes to zero as $r\to+\infty$ then necessarily
$\lim_{r\rightarrow +\infty}\frac{\psi'(r)}{\psi(r)}=0$ (see Lemma
\ref{eq:lambda_1>0}) and the previous comment applies, whatever the rate of decay of the curvatures is. Hence, in
this case no spectral gap is present and the expected
picture is of Euclidean type, but we shall not address this issue
here.

Under the above mentioned assumptions on $\psi$, we prove in Theorem \ref{exist gr st} existence of a
finite energy radial solution to \eqref{LEF} in the subcritical
case $p\in\left(1,\frac{n+2}{n-2}\right)$. Uniqueness of such
solution holds under a further technical condition on $\psi$, see
Theorem \ref{subcr set}. In the supercritical range
$p\ge\frac{n+2}{n-2}$, we prove in Theorem \ref{teo} that if a
suitable power of the volume of geodesic balls is \it convex \rm
as a function of $r$, all local radial solutions to \eqref{LEF}
are everywhere positive and no solution to the Dirichlet problem
on geodesic balls exists. In particular, such results hold if
$\psi$ itself is convex.

In both subcritical and supercritical cases, we provide an exact
description of the asymptotic behavior of positive radial
solutions of \eqref{LEF}. In Theorem \ref{asym sub} we show that,
in the subcritical case, solutions in the energy space $H^1(M)$
have a fast decay to zero which can be characterized explicitly in
terms of the function $\psi$. An interesting phenomenon occurs for
solutions which do not belong to $H^1(M)$: they admit a limit as
$r\to +\infty$ which can be strictly positive or equal to zero
depending on the integrability at infinity of the function
$\psi/\psi'$.

The same phenomenon occurs in the supercritical case as shown in
Theorem \ref{asym}.

The second part of this paper is devoted to stability of
solutions. Here by stability we mean the so-called linearized
stability. Namely, we say that a solution $u$ of
\eqref{LEF} is stable if the quadratic form associated with the
linearized operator at $u$ is nonnegative definite. Stability of
solutions of nonlinear equations in the whole euclidean space is a
widely studied problem, especially in the case of the
Lane-Emden-Fowler equation and  of the Gelfand equation $-\Delta
u=e^u$, see e.g. \cite{cc,dafa,dufa,farina,farina2,fsv} and
references therein. See also \cite{bv} for results on
stability of the Lane-Emden-Fowler and Gelfand equations in
bounded domains.

In order to localize the instability of certain solutions we shall also
study the stability of solutions outside a compact set, see e.g. \cite{dafa,farina,farina2}.

Since the cut locus of the pole $o$ is empty by assumption, any
Riemannian model $M$ we are considering is diffeomorphic to
${\mathbb R}^n$, and the main purpose of the present paper is to
understand which is the role of the curvature properties of $M$ in
determining stability of solutions of \eqref{LEF}, in particular when sectional curvatures are \it negative\rm. We comment here that the existence of stable solutions to semilinear elliptic equations when Ricci curvature is \it positive \rm has consequences on the structure of the manifold itself (and on the solution as well), as shown recently in \cite{FMV}.

For completeness we first recall what happens when $M$ is the $n$-dimensional
euclidean space. From \cite{farina}, we know that no nontrivial stable solution
(also nonradial) exists if $n\leq 10$ or $n\geq 11$ and
$p<p_c(n)=\frac{(n-2)^2-4n+8\sqrt{n-1}}{(n-2)(n-10)}$, where
$p_c(n)>\frac{n+2}{n-2}$ is the so-called Joseph-Lundgren
exponent, see \cite{jl}. On the other hand, for $n\geq 11$ and
every $p\geq p_c(n)$ there exists a positive radial stable
solution, see \cite{farina,jl}.

Also we note that when $n\leq 10$ or $n\geq 11$ and $p< p_c(n)$,
with $p\neq \frac{n+2}{n-2}$, the euclidean equation admits no
nontrivial solution which is stable outside a compact set. On the
other hand, if $p=\frac{n+2}{n-2}$ then the euclidean equation
admits solutions in $H^1(\R^n)$ which are stable outside a
compact set. Among them there are the well-known
one-parameter family of solutions of \eqref{LEF} which achieve the
best Sobolev constant in $\R^n$.

As we said before, under suitable assumptions on $\psi$, a
Poincar\'e type inequality holds. The validity of this inequality is strictly related to the existence of stable
solutions. In Theorems \ref{stability-1}-\ref{stability-2} we
prove that stable radial solutions of \eqref{LEF} always exist in
any dimension and for any $p>1$ provided that their value at the
origin is small enough.

This phenomenon is deeply in contrast with the euclidean case
where the existence of nontrivial radial stable solutions only
depends on $n$ and $p$ but not on the value of the solution at the
origin. We also recall that, thanks to rescaling invariance
properties of the Lane-Emden-Fowler equation, in the euclidean case all
nontrivial radial solutions may be represented as a one-parameter
family of rescaled functions. This property explains why there is no dependence of the stability on the value at the
origin.

The next step is to understand if radial stable solutions also
exist for larger values at the origin. Our main results on
stability, Theorem \ref{stability-1}-\ref{stability-2}, state that
independently of the dimension $n$ and of the power $p$, the set
$\mathcal S$ of the values at the origin for which the
corresponding radial solution of \eqref{LEF} is stable, is a
closed interval containing $0$. One may ask if this interval
coincides with $[0,+\infty)$. In Theorem \ref{stability-1}
we show that, under the same assumptions on $n$ and $p$ for which
in the euclidean case we have nonexistence of nontrivial stable
solutions, in our Riemannian model the set $\mathcal S$ is a bounded closed interval.

This result, which shows instability of radial solutions with a
large value at the origin, is based on a blow-up argument which
has as a limit problem the Lane-Emden-Fowler equation in the
euclidean space. This justifies the relationship between assumptions of Theorem \ref{stability-1} and
the nonexistence result of stable solutions in the euclidean case.

It is left as an open question to understand if the assumptions of
Theorem \ref{stability-1} are also necessary for boundedness of
the set $\mathcal S$.

Stability properties are strictly related to ordering of radial
solutions of \eqref{LEF}. Indeed in Theorem \ref{t:ordering} we
prove that radial solutions of \eqref{LEF} corresponding to values
at the origin in the set $\mathcal S$ are ordered.

Finally, in Theorem \ref{outside-compact-set} we show that all
radial solutions of \eqref{LEF} are stable outside compact sets
independently of $n\ge 3$ and $p>1$, provided that
$\psi/\psi'\not\in L^(0,\infty)$.

This paper is organized as follows: in Section \ref{s:main} we put the assumptions and the statements of the main results while Sections \ref{super}-\ref{stabil} are devoted to the proofs.

\subsection{Notation and preliminaries}\label{notation} The $C^2$ smoothness of $M$ around $o$ implies
that $\psi$ must be extendible to $r=0$ with the extension, still
denoted by $\psi$, satisfying $\psi(0)=\psi^{\prime\prime}(0)=0$, $\psi^\prime(0)=1$, the
prime indicating right derivative. In greater generality, a power
series for $\psi$ near $r=0$ must contain only odd powers of $r$ should one require additional smoothness at $o$,
see e.g. \cite{Sto}, pp. 179-183, and also \cite{GW}.

The Riemannian Laplacian of a scalar function $f$ on $M$ is given, in the above coordinates, by
\[\begin{split}
\Delta_g
f(r,\theta_1,\ldots,\theta_{n-1})&=\frac1{(\psi(r))^{n-1}}\frac{\partial}{\partial
r}\left[(\psi(r))^{n-1} \frac{\partial f}{\partial
r}(r,\theta_1,\ldots,\theta_{n-1})
\right]\\
&+\frac1{(\psi(r))^2}\Delta_{{\mathbb
S}^{n-1}}f(r,\theta_1,\ldots,\theta_{n-1}),
\end{split}
\]
where $\Delta_{{\mathbb S}^{n-1}}$ is the Riemannian Laplacian on the unit sphere ${\mathbb S}^{n-1}$.
In particular, for \it radial \rm functions, namely functions depending only on $r$, one has
\[
\Delta_g f(r)=\frac1{(\psi(r))^{n-1}}\left[(\psi(r))^{n-1}
f^\prime(r)\right]^\prime=f^{\prime\prime}(r)+(n-1)\frac{\psi^\prime(r)}{\psi(r)}f^\prime(r),
\]
where from now on a prime will denote, for radial functions,
derivative w.r.t. $r$. Notice that the quantity $(n-1)\frac{\psi^\prime(r)}{\psi(r)}$ has a geometrical meaning,
namely it represents mean curvature of the geodesic sphere of radius $r$ in the radial direction.
Let $\omega_n$ be the volume of the $n$-dimensional unit sphere. Then
\[
S(r)=\omega_n (\psi(r))^{n-1},\ \ \ V(r)=\int_0^rS(t)\,{\rm d}t=\omega_n\int_0^r(\psi(t))^{n-1}\,{\rm d}t
\]
represent, respectively, the area of the geodesic sphere $\partial B(o,r)$ and the volume of the geodesic ball $B(o,r)$.
Moreover (see e.g. \cite{Besse}, \cite{GW}) one can show that
\[
\frac1{n-1}\rm Ric\it\, (\partial r,\partial
r)=K_\pi(r)=-\frac{\psi^{\prime\prime}(r)}{\psi(r)},
\]
where $\rm Ric\it\, (\partial r,\partial r)$ is the Ricci tensor in the radial direction, and $K_\pi(r)$ denotes
sectional curvatures w.r.t planes containing $\partial r$. One shows also that the sectional curvatures w.r.t.
planes orthogonal to $\partial r$ is given by $\frac{1-(\psi^\prime(r))^2}{(\psi(r))^2}$. Sectional curvatures equal -1 on the hyperbolic space, whereas they are still negative, but growing in modulus when for example one has, for large $r$, $\psi(r)=e^{r^a}$ for some $a>1$, a case which can be covered by most of our results.

We consider radial solutions to the
Lane-Emden-Fowler equation \eqref{LEF}. Radial \it local \rm solutions near
$r=0$ to \eqref{LEF} with $u(0)=\alpha\neq 0$ exist, are unique
and satisfy the Cauchy problem
\begin{equation}\label{cauchy}
\left \{ \begin{array}{ll}
-\displaystyle{\frac1{(\psi(r))^{n-1}}\left[(\psi(r))^{n-1}u^\prime(r)\right]^\prime}=|u(r)|^{p-1}u(r) & (r>0)\\ \\
u(0)=\alpha\quad u'(0)=0\,. &
\end{array}
\right.
\end{equation}
For any $r>0$, let us denote by $u_\alpha(r)$ or by $u(\alpha,r)$
the unique solution of the Cauchy problem \eqref{cauchy}.

\section{Assumptions and main results} \label{s:main}

Let $\psi$ be the function defined in the introduction. Let us
introduce the following assumptions on $\psi$:
\begin{itemize}
\item[$(H_1)$] $\psi\in C^2([0,+\infty))$: $\psi(0)=\psi^{\prime\prime}(0)=0$ and
$\psi'(0)=1$ ;

\item[$(H_2)$] $\psi'(r)\geq 0$ for every $r>0$ \, ,

\item[$(H_3)$] $l:=\displaystyle{\liminf_{r\rightarrow
+\infty}\frac{\psi'(r)}{\psi(r)}>0}$ \,.
\end{itemize}

Assumption $(H_1)$ is necessary to make the geometric setting
outlined in Section \ref{notation} consistent. Assumptions
$(H_2)-(H_3)$ are sufficient conditions to guarantee positivity of
bottom of the $L^2$ spectrum of $-\Delta_g$ in $M$, see Lemma
\ref{eq:lambda_1>0}. Throughout this paper we denote the bottom of
the $L^2$ spectrum of $-\Delta_g$ by $\lambda_1(M)$. Under
assumptions $(H_1)-(H_3)$ one can show easily that every solution of
\eqref{cauchy} is global.
\begin{pro}\label{global}
Let $p>1$ and assume that $\psi$ satisfies assumptions
$(H_1)-(H_3)$. Then, for any $\alpha\neq 0$ the local solution to
\eqref{cauchy} may be continued for all $r>0$,
$\displaystyle{\lim_{r\rightarrow +\infty} u'(r)=0}$ and
$\displaystyle{\lim_{r\rightarrow +\infty} u(r)}$ exists and is
finite. In particular \eqref{LEF} admits infinitely many
nontrivial radial solutions.
\end{pro}
Since the proof of Proposition \ref{global} can be achieved following the lines of that of \cite[Lemma 4.1]{BGGV}, we omit it. The same proof does not work if $l=0$ in $(H_3)$. However, if $\psi$ satisfies
\begin{equation} \label{serrin}
\exists \,\beta, \beta'>0: \quad\frac{\beta}{r}\leq\frac{\psi'(r)}{\psi(r)}\leq \beta'\quad \forall r\geq r_0
\end{equation}
for some $r_0>0$ one may repeat the proof of
\cite[Theorem 5]{pucci} to show that
$\displaystyle{\lim_{r\rightarrow +\infty}
u(r)=0}=\displaystyle{\lim_{r\rightarrow +\infty} u'(r)}$. Clearly,
\eq{serrin} includes the euclidean case $\psi(r)=r$ but does not
hold if, for instance, $\psi(r)=\log(r)$.\bigskip\par

The results concerning existence and qualitative behavior of
solutions of \eqref{LEF} are strongly influenced by the range in
which the power $p$ varies. In the sequel we distinguish the
subcritical case $1<p<2^*-1=\frac{n+2}{n-2}$ and the supercritical
case $p\ge \frac{n+2}{n-2}$.

\bigskip

\textsc{$\bullet$ The subcritical case.}
Let start with the following existence result of a radial $H^1(M)$-solution
of \eqref{LEF}:

\begin{thm}\label{exist gr st}
Let $1<p< \frac{n+2}{n-2}$ and $\psi$ satisfy assumptions
$(H_1)-(H_3)$. Then \eqref{LEF} admits a positive radial solution
$u\in H^1(M)$.
\end{thm}

One may wonder if \eqref{LEF}
admits a unique radial solution belonging to $H^1(M)$. This happens in the hyperbolic space, i.e. $\psi(r)=\sinh(r)$, see
\cite{mancini}. In order to guarantee uniqueness of radial
$H^1(M)$-solutions, we introduce a supplementary condition on the
function $\psi$. To this purpose we recall from \cite{kwong} the
following definition:
\begin{den}
A function $G:(0,+\infty)\rightarrow \R$ differentiable, satisfies
the $\Lambda-$property if there exists $0\leq r_1\leq +\infty$
such that $G'\geq 0$ in $(0,r_1)$ and $G'\leq 0$ in
$(r_1,+\infty)$ with $G'\not \equiv 0$.
\end{den}
Note that the definition includes the cases in which $G$ is always
nondecreasing or nonincreasing in $[0,+\infty)$.
We are ready to state the following uniqueness result:

\begin{thm}\label{subcr set}
Let $1<p<\frac{n+2}{n-2}$. Assume that $\psi$ satisfies
$(H_1)-(H_2)$ and that there exists
\begin{equation}\label{l}
\displaystyle{\lim_{r\rightarrow
+\infty}\frac{\psi'(r)}{\psi(r)}}=\displaystyle{\lim_{r\rightarrow
+\infty}\frac{\psi''(r)}{\psi'(r)}=l\in\left(0,+\infty\right]}\,.
\end{equation}
Furthermore, set $\delta:=\frac{2(n-1)}{p+3}$ and let the function
$$G(r):=\delta \,\psi^{\delta(p-1)-2}(r)\left[(\delta+2-n)\left(\psi'(r)\right)^{2}-\psi''(r)\psi(r) \right]\quad (r>0)$$
satisfy the $\Lambda-$property. \par Finally, if $l=+\infty$
assume that $\psi$ satisfies the extra condition
\begin{equation}\label{extra}
\frac{\psi'(r)}{\psi(r)}=o( \psi^{\delta}(r))\, , \qquad
\frac{\psi''(r)}{\psi'(r)}=o( \psi^{\delta}(r)) \quad \text{as
}r\rightarrow +\infty\,.
\end{equation}
Then, problem \eqref{cauchy} admits a unique positive solution $U$ belonging to
$H^1(M)$. Moreover, every solution to \eqref{cauchy} with
$0<\alpha<U(0)$ is of one sign, while any solution to
\eqref{cauchy} with $\alpha>U(0)$ is sign-changing.
\end{thm}

Concerning the validity of the $\Lambda-$property for the function
$G$ defined in Theorem \ref{subcr set} we observe that it is
satisfied when $\psi(r)=\sinh(r)$, i.e. $M=\mathbb H^n$. For more
general function $\psi$ we state the following
\begin{pro} \label{remark2.5}
If $\psi$ satisfies assumptions $(H_1)-(H_3)$, and in addition $\psi$ is four times differentiable with
$\psi'''(r)>0$ and $\left(\frac{\psi'(r)}{\psi'''(r)}\right)'\leq
0$ for every $r>0$, then the function $G$ defined in Theorem \ref{subcr set} satisfies the $\Lambda-$property for
every $\frac{2n+1}{2n-3}\leq p <\frac{n+2}{n-2}$.
\end{pro}
By Proposition \ref{remark2.5}, it follows that if
$\psi(r)=re^{r^{2\gamma}}$ then the corresponding function $G$
satisfies the $\Lambda-$property for every $\gamma \geq 1$.

Concerning condition \eqref{extra}, we observe that it holds, for instance, if $\frac{\psi'(r)}{\psi(r)}=P(r)$
eventually, where $P$ is a nonconstant polynomial.

Finally we state a result dealing with the asymptotic behavior of
radial positive solutions of \eqref{LEF}.

\begin{thm}\label{asym sub} Let $n\ge 3$ and $1<p < \frac{n+2}{n-2}$. Suppose that $\psi$ satisfies
assumptions $(H_1)-(H_2)$ and
\begin{equation} \label{eq:ex-lim}
\lim_{r\to +\infty} \frac{\psi'(r)}{\psi(r)}=l\in (0,+\infty] \, .
\end{equation}
Finally in the case $l=+\infty$ assume the supplementary condition
\begin{equation} \label{eq:supplementary}
\left[\log\left(\frac{\psi'(r)}{\psi(r)} \right)\right]'=O(1)
\qquad \text{as } r\to +\infty \, .
\end{equation}

Let $u$ be a radial positive solution of \eqref{LEF}.
\begin{itemize}
\item[(i)] If $u\in H^1(M)$ then there exists $L\in (-\infty,0)$
such that
\begin{equation*} 
\lim_{r\rightarrow +\infty}\psi^{n-1}(r) u'(r)=L \, .
\end{equation*}
Moreover
\begin{equation*} 
\lim_{r\rightarrow +\infty}\psi^{n-1}(r) u(r)=\frac{|L|}{(n-1)l}
\qquad \text{if } l<+\infty \, ,
\end{equation*}
and
\begin{equation*} 
\lim_{r\rightarrow
+\infty}\frac{u(r)}{\int_{r}^{+\infty}\psi^{1-n}(s)\,ds}=|L|
\qquad \text{if } l=+\infty \, .
\end{equation*}

\item[(ii)] If $u\not\in H^1(M)$ and $\frac{\psi}{\psi'}\in
L^1(0,\infty)$ then
$$
\lim_{r\to +\infty} u(r)\in (0,+\infty) \, .
$$
\item[(iii)] If $u\not\in H^1(M)$ and $\frac{\psi}{\psi'}\not\in
L^1(0,\infty)$ then $u$ vanishes at infinity with the following
rate
\begin{equation*}
\lim_{r\to +\infty} \left(\int_{0}^r \frac{\psi(s)}{\psi'(s)}\,
ds\right)^{1/(p-1)} u(r)=\left(\frac{n-1}{p-1}\right)^{1/(p-1)} \,
.
\end{equation*}
In particular when $l<+\infty$ we have
\begin{equation*}
\lim_{r\to +\infty} r^{1/(p-1)}
u(r)=\left(\frac{l(n-1)}{p-1}\right)^{1/(p-1)} \, .
\end{equation*}
\end{itemize}
\end{thm}

\bigskip

\textsc{$\bullet$ The supercritical case.} Throughout this paper let us denote by $B_R$ the geodesic
ball centered at $o$ of radius $R$, i.e.
$$
B_R:=\{(r,\theta_1,\ldots,\theta_{n-1}): 0<r<R \text{ and } \theta_1,\ldots,\theta_{n-1}\in{\mathbb S}^{n-1}\}\,.
$$

\begin{thm}\label{teo}
Let $p\geq \frac{n+2}{n-2}$ and $\psi$ satisfy assumptions $(H_1)-(H_2)$. If $p=\frac{n+2}{n-2}$ assume
furthermore that $\psi$ is three times differentiable near $0$ with $\psi''(0)=0$ and $\psi'''(0)>0$. Finally, let the function
$$
A(r):=\left(\int_0^r(\psi(s))^{n-1}\,ds\right)^{\frac{p-1}{2(p+1)}}=c\,\left[{\rm Vol}\, B(o,r)\right]^{\frac{p-1}{2(p+1)}}
$$
be convex on $[0,+\infty)$. Then any solution $u(r)$ to \eqref{cauchy} does not change sign for all $r\in[0,+\infty)$. In particular, the Dirichlet problem
$$ \left \{
 \begin{array}{ll}
-\Delta_g u=|u|^{p-1}u & \text{in }B_{\overline{r}}\\
u=0 &  \text{on }\partial B_{\overline{r}}
\end{array}
\right.$$ with $0<\overline{r}<+\infty$ has no nontrivial radial
solutions.
\end{thm}
Concerning the convexity of the function $A$ defined in Theorem \ref{teo} we state the following
\begin{pro} \label{convexity-A(r)} Assume that $\psi$ satisfies $(H_1)-(H_2)$. Let $A$ be the function defined in Theorem \ref{teo}. Then we have:
\begin{itemize}
\item[(i)] if $\psi$ is convex, then $A$ is also convex;

\item[(ii)] if $\psi$ is such that \eq{eq:ex-lim} holds with $l<+\infty$, then $A$
is eventually convex at $+\infty$;

\item[(iii)] if $\psi$ is such that \eq{eq:ex-lim} holds with $l=+\infty$ and
\eqref{eq:supplementary} is satisfied, then $A$ is eventually
convex at $+\infty$.
\end{itemize}
\end{pro}

From Proposition \ref{convexity-A(r)} it follows that the assumptions
of Theorem \ref{teo} are satisfied either by the hyperbolic model (see \cite{BGGV}) or by models having unbounded negative sectional curvatures such as $\psi(r)=re^{r^{2\gamma}}$ with
$\gamma\ge 0$.

Similarly to the subcritical case, for the
asymptotic behavior of radial positive solutions of \eqref{LEF} we have

\begin{thm}\label{asym} Let $n\ge 3$ and $p\ge \frac{n+2}{n-2}$. Suppose that $\psi$ satisfies
assumptions $(H_1)-(H_2)$, \eqref{eq:ex-lim} and that the function
$A=A(r)$ defined in Theorem \ref{teo} is convex. Finally in the
case $l=+\infty$ we also assume \eqref{eq:supplementary}. Let $u$
be a radial (positive) solution of \eqref{LEF}.
\begin{itemize}
\item[(i)] If $\frac{\psi}{\psi'}\in L^1(0,\infty)$ then
$$
\lim_{r\to +\infty} u(r)\in (0,+\infty) \, .
$$
\item[(ii)] If $\frac{\psi}{\psi'}\not\in L^1(0,\infty)$ then $u$
vanishes at infinity with the following rate
\begin{equation*}
\lim_{r\to +\infty} \left(\int_{0}^r \frac{\psi(s)}{\psi'(s)}\,
ds\right)^{1/(p-1)} u(r)=\left(\frac{n-1}{p-1}\right)^{1/(p-1)} \,
.
\end{equation*}
In particular when $l<+\infty$ we have
\begin{equation*}
\lim_{r\to +\infty} r^{1/(p-1)}
u(r)=\left(\frac{l(n-1)}{p-1}\right)^{1/(p-1)} \, .
\end{equation*}
\end{itemize}
\end{thm}

\bigskip

\textsc{$\bullet$ Stability of radial solutions of \eqref{LEF}.}
We start by explaining what we mean by stability and stability outside a compact set, see also \cite{fsv}.

\begin{den}\label{def1}
A solution $u\in C^2(M)$ to $\eq{LEF}$ is stable if
\begin{equation}\label{stab}
\int_{M} |\nabla_g \varphi|_{g}^2\,dV_{g}- p\int_{M} |u|^{p-1}
\varphi^2\,dV_g\geq 0 \quad \forall \, \varphi\in C_c^{\infty}(M).
\end{equation}
A solution $u\in C^2(M)$ to $\eq{LEF}$ is stable outside the compact set $K$ if
\begin{equation}\label{stabK}
\int_{M\setminus K}|\nabla_g \varphi|_{g}^2\,dV_{g}-
p\int_{M\setminus K} |u|^{p-1} \varphi^2\,dV_g\geq 0 \quad \forall
\, \varphi\in C_c^{\infty}(M\setminus K).
\end{equation}
\end{den}
For any $n\ge 11$, let $p_c(n)=\frac{(n-2)^2-4n+8\sqrt{n-1}}{(n-2)(n-10)}$ be the Joseph-Lundgren
exponent. We can now state the first result concerning stability of radial solutions of \eqref{LEF}.

\begin{thm} \label{stability-1}
Let $3\le n\le 10$ and $p>1$ or $n\ge 11$ and $1<p<p_c(n)$. Assume
that $\psi$ satisfies $(H_1)-(H_3)$. For any $\alpha\ge 0$ denote
by $u_\alpha$ the unique solution of \eqref{cauchy}. There exists
$\alpha_0\in (0,+\infty)$ such that
\begin{itemize}
\item[(i)] if $\alpha\in [0,\alpha_0]$ then $u_\alpha$ is stable;

\item[(ii)] if $\alpha>\alpha_0$ then $u_\alpha$ is unstable.
\end{itemize}
Furthermore we also have $\alpha_0\ge \left(p^{-1} \lambda_1(M)
\right)^{1/(p-1)}$. The inequality is strict if one of the following alternatives hold
\begin{equation}\label{alternatives}\displaystyle{\limsup_{r\rightarrow +\infty}\frac{\psi'(r)}{\psi(r)}<+\infty}\qquad \text{or} \qquad l=+\infty \text{ in } (H_3) \text{ and } \psi \text{ satisfies } \eqref{eq:supplementary} \text{ and } \psi/\psi'\not\in L^1(0,\infty)\,.
\end{equation}
\end{thm}

By comparing Theorem \ref{stability-1} with the stability result in the euclidean case, one sees that the existence of stable solutions in
dimension $n\le 10$ or in dimension $n\ge 11$ but with $p<p_c(n)$,
seems to be strictly related to the validity of the Poincar\'e
inequality (see Table 1 and Table 2 below). Indeed the existence
of the positive number $\alpha_0$ introduced in Theorem
\ref{stability-1} comes from the the positivity of the bottom of the $L^2$ spectrum $\lambda_1(M)$
of $-\Delta_g$ in $M$ as one can see
from the estimate $\alpha_0>\left(p^{-1} \lambda_1(M)
\right)^{1/(p-1)}$, see also Lemma \ref{eq:lambda_1>0}. On the contrary, in the euclidean case the
Poincar\'e inequality in $\R^n$ does not hold and, if $n\le 10$ or $n\ge 11$ but $p<p_c(n)$,
all nontrivial solutions of the Lane-Emden-Fowler equation are unstable.

We observe that the assumptions on the dimension $n$ and on the
power $p$ in Theorem \ref{stability-1} are at least sufficient to
show the existence of the switch between stability for small
values of $\alpha$ and instability for large values of $\alpha$
but it is not clear if they are also necessary. As a partial
result we state the validity of the following alternatives:

\begin{thm} \label{stability-2}
Let $n\ge 11$ and $p\ge p_c(n)$. Assume that $\psi$ satisfies
$(H_1)-(H_3)$. For any $\alpha\ge 0$ denote by $u_\alpha$ the
unique solution of \eqref{cauchy}. There exists $\alpha_0\in
(0,+\infty]$ such that either $\alpha_0=+\infty$ and $u_\alpha$ is
stable for any $\alpha\ge 0$ or $\alpha_0<+\infty$ and $u_\alpha$
is stable for any $\alpha\in [0,\alpha_0]$ and unstable for any
$\alpha>\alpha_0$.
\end{thm}
Concerning stability of solutions in the energy space, we state
the following

\begin{pro} \label{p:liouville-type}
Let $n\ge 3$ and $p>1$. Assume that $\psi$ satisfies
$(H_1)-(H_3)$. Let $u$ be a radial stable solution of \eqref{LEF}.
If $u\in L^2(M)$ then $u\equiv 0$.
\end{pro}

Stability properties of solutions are related to ordering and
intersection properties of radial solutions of \eqref{LEF}:

\begin{thm} \label{t:ordering}
Let $n\ge 3$ and $p>1$. Assume that $\psi$ satisfies
$(H_1)-(H_3)$. Let $\alpha,\beta\ge 0$ and let $u_\alpha,u_\beta$
be the corresponding solutions of \eqref{cauchy}. If $u_\alpha$
and $u_\beta$ are stable then they do not intersect. In particular
stable solutions are strictly positive (or strictly negative) and
if $\alpha_0\in (0,+\infty]$ is as in Theorems
\ref{stability-1}-\ref{stability-2} then all solutions in the set
$\{u_\alpha:\alpha\in [0,\alpha_0)\}$ are ordered.
\end{thm}
We conclude the section by dealing with stability
outside a compact set.

\begin{thm} \label{outside-compact-set}Let $n\ge 3$ and $p>1$. Assume that $\psi$ satisfies
$(H_1)-(H_3)$. Then any radial solution of \eqref{LEF} is
stable outside a compact set provided that \eqref{alternatives} holds.
\end{thm}
Differently from the euclidean case, see Table 1 below, Theorem
\ref{outside-compact-set} states that under assumptions
$(H_1)-(H_3)$ and \eq{alternatives}, all solutions
of \eqref{cauchy} are stable outside a compact independently of
the value of the power $p$. We note that assumption \eq{alternatives} assures that solutions of \eqref{cauchy} vanish as $r\to +\infty$ (see Proposition \ref{global}, formula \eq{serrin}, Theorems \ref{asym sub} and \ref{asym}). Hence, the difference from the euclidean case once more comes from the
fact that, under assumptions $(H_1)-(H_3)$, the bottom of the
$L^2$ spectrum of $-\Delta_g$ in $M$ is strictly positive.

\begin{table}[h] 
\begin{center}
\begin{tabular}{|c|c|c|}
\hline & $n\leq 10 \ {\rm or} \ (n\geq 11$ and $p<p_c(n))$ & $
n\geq 11$ and $p\geq p_c(n)$ \\ \hline $u_{\alpha}$ stable
$\forall$ $\alpha \neq 0$  & NO                     & YES
\\ \hline $u_{\alpha}$ unstable $\forall$
$\alpha \neq 0$  & YES                     & NO
 \\ \hline
$u_{\alpha}$ stable outside a compact $\forall$ $\alpha$  & NO if
$p\neq \frac{n+2}{n-2}$ YES if $p= \frac{n+2}{n-2}$ & YES   \\
\hline
\end{tabular}
\caption{Stability of solutions $u_{\alpha}$ to \eqref{cauchy}
when $\psi(r)=r$ (Euclidean case).}
\end{center}
\end{table}

\begin{table}[h] 
\begin{center}
\begin{tabular}{|c|c|c|}
\hline & $n\leq 10 \ {\rm or} \ (n\geq 11$ and $p<p_c(n))$ & $
n\geq 11$ and $p\geq p_c(n)$ \\ \hline $u_{\alpha}$ stable
$\forall$ $0<|\alpha|\leq \alpha_0$  & YES                     & YES
if $|\alpha|< \alpha_0$  \\ \hline $u_{\alpha}$ unstable $\forall$
$|\alpha|> \alpha_0$  & YES                     & ?  \\ \hline
$u_{\alpha}$ stable outside a compact $\forall$ $\alpha$  & YES if \eq{alternatives} holds &
YES if \eq{alternatives} holds
\\ \hline
\end{tabular}
\caption{Stability of solutions $u_{\alpha}$ to \eqref{cauchy} for
$\psi$ satisfying $(H_1)-(H_3)$.}
\end{center}
\end{table}

\section{Proof of the results in the supercritical case}\label{super}

\subsection{Proof of Theorem \ref{teo}}
Let $u$ be a nontrivial solution of \eqref{cauchy}, up to
replace $u$ with $-u$, we may assume $\alpha>0$. For $r\geq 0$, we set
$$
P(r):=\left[(p+1)\int_0^r(\psi(s))^{n-1}\,ds\right]\left(\frac{(u^\prime(r))^2}2+\frac{|u(r)|^{p+1}}{p+1}\right)+(\psi(r))^{n-1}u(r)u^\prime(r)\,.
$$
Then, for $u$ solving \eq{cauchy} we get
$$
P^\prime(r)=\left[\frac{p+3}2(\psi(r))^{n-1}-(n-1)(p+1)\frac{\psi^\prime(r)}{\psi(r)}\int_0^r(\psi(s))^{n-1}\,ds\right](u^\prime(r))^2:=K(r)(u^\prime(r))^2,
$$
the latter equality being meant as a definition of $K(r)$. We
notice that, as $r\downarrow0$, the known asymptotics of $\psi(r)$
as $r\to0$ implies that $K(r)\sim r^{n-1}[(n+2)-(n-2)p]/(2n)$ if
$p>\frac{n+2}{n-2}$ and $K(r)\sim
r^{n+1}[-2(n-1)/(n^2-4)]\,\psi'''(0)$, if $p=\frac{n+2}{n-2}$,
where we exploit the assumptions $\psi''(0)=0$ and $\psi'''(0)>0$.
\par
This clearly shows that, in such range of $p$, $K(r)<0$ for $r$
sufficiently small, and hence that $P^\prime(r)<0$ for the same
values of $r$. The strict inequality follows from the fact
that $u^\prime(r)\not=0$ for $r\in(0,\varepsilon)$ for a suitable
$\varepsilon>0$, a fact which holds since $u$ is different from
zero in a right neighborhood of zero and by \eqref{cauchy} we have
\begin{equation*}
u'(r)=-\frac{1}{(\psi(r))^{n-1}} \int_0^r
(\psi(s))^{n-1}|u(s)|^{p-1}u(s)\, ds \, .
\end{equation*}
Hence, since $P(0)=0$, we have proven that $P(r)<0$ in a
sufficiently small right neighborhood of zero. We claim that
$K(r)\le 0$ for any $r>0$ which implies $P$ nonincreasing in
$(0,+\infty)$; being $P(r)<0$ for $r>0$ small enough this yields
$P(r)<0$ for any $r>0$.

Let us prove the claim. Let
$\Psi(r):=\int_0^r(\psi(s))^{n-1}\,ds$. One computes
\[
\frac{\Psi^{\prime\prime}(r)}{\Psi^\prime(r)}=(n-1)\frac{\psi^\prime(r)}{\psi(r)}.
\]
Hence, requiring that $K(r)\le0$ is equivalent to ask that
\[
\frac{\Psi^{\prime\prime}(r)}{\Psi^\prime(r)}\ge
\frac{p+3}{2(p+1)}\frac{\Psi^\prime(r)}{\Psi(r)},
\]
where we have used the fact that $\Psi(r)>0$ for all
$r\in(0,\infty)$. Recall that, by construction, $\Psi^\prime(r)>0$
for all $r>0$. Setting $a_p=\frac{p+3}{2(p+1)}$ we can then
rewrite the latter formula as
\[
\left[\log\left(\frac{\Psi^\prime(r)}{(\Psi(r))^{a_p}}\right)\right]^\prime\ge0.
\]
or equivalently, setting $c_p=1-a_p=\frac{p-1}{2(p+1)}$, as
\[
\left[\log\left(\left((\Psi(r))^{c_p}\right)^\prime\right)\right]^\prime\ge0.
\]
The latter condition is clearly equivalent to
$\left[(\Psi(r))^{c_p}\right]^{\prime\prime}\ge0$, namely to the
fact that $(\Psi(r))^{c_p}$ is convex (recall that $\psi$ is at
least $C^2$). This completes the proof of the claim. Since
$u(0)>0$, if we assume that there exists $\rho>0$ such that
$u(\rho)=0$ then we have $u'(\rho)<0$ and hence $P(\rho)>0$, a
contradiction.

\subsection{Proof of Proposition \ref{convexity-A(r)}} A simple computation yields that $A(r)$ is convex if and only if
the function
$h(r):=2(n-1)(p+1)\psi^\prime(r)\int_0^r\psi^{n-1}(s)\,ds-(p+3)\psi^n(r)$
is positive in $(0,+\infty)$. This readily follows if $\psi$ is a
convex function too. Indeed, we have $h(0)=0$ and  $$
h'(r)=(p(n-2)-(n+2))\psi'(r)\psi^{n-1}(r)+2(n-1)(p+1)\psi''(r)\int_0^r(\psi(s))^{n-1}\,ds\,.$$
Hence, statement $(i)$ follows.
\par
Then we turn to the proofs of $(ii)$ and $(iii)$. First we claim that
\begin{equation}\label{lim}
\displaystyle{\lim_{r\rightarrow +\infty}
(n-1)\,\psi^\prime(r)\,\frac{\int_0^r\psi^{n-1}(s)\,ds}{\psi^n(r)}=1}\,.
\end{equation}
By this,
$$\displaystyle{\lim_{r\rightarrow +\infty} h(r)}= \displaystyle{\lim_{r\rightarrow +\infty} \psi^{n}(r)\left[2(n-1)(p+1)\psi^\prime(r)\,\frac{\int_0^r\psi^{n-1}(s)\,ds}{\psi^n(r)}-(p+3)\,\right]=+\infty}$$
and we conclude.\par Next we prove $\eq{lim}$. If
$\displaystyle{\lim_{r\rightarrow
+\infty}\frac{\psi'(r)}{\psi(r)}=l}$ for some $0<l<+\infty$, the
claim readily follows by the l'H\^{o}pital rule. Indeed, we have
$$\displaystyle{\lim_{r\rightarrow +\infty} \frac{\int_0^r\psi^{n-1}(s)\,ds}{\psi^{n-1}(r)}=\lim_{r\rightarrow +\infty} \frac{\psi(r)}{(n-1)\psi'(r)}=\frac{1}{(n-1)l}}\,.$$
Let now $l=+\infty$. Again, by the l'H\^{o}pital rule we deduce
$$\displaystyle{\lim_{r\rightarrow +\infty} (n-1)\,\psi^\prime(r)\,\frac{\int_0^r\psi^{n-1}(s)\,ds}{\psi^n(r)}=(n-1)\lim_{r\rightarrow +\infty}\frac{\left[ \frac{\psi'(r)}{\psi(r)}\int_0^r\psi^{n-1}(s)\,ds\right]'}{[\psi^{n-1}(r)]'}}$$
$$\displaystyle{=1+\lim_{r\rightarrow +\infty}\left[\log\left(\frac{\psi'(r)}{\psi(r)} \right)\right]'\frac{\int_0^r\psi^{n-1}(s)\,ds}{\psi^{n-1}(r)}}\,.$$
Then, since $\int_0^r\psi^{n-1}(s)\,ds=o(\psi^{n-1}(r))$ as
$r\rightarrow +\infty$, $\eq{lim}$ holds for every function $\psi$
such that $\left[\log\left(\frac{\psi'(r)}{\psi(r)}
\right)\right]'$ remains bounded.

\subsection{Proof of Theorem \ref{asym}} We start with the
following estimate from below on solutions of \eqref{cauchy}:

\begin{lem}\label{no exp dec}
Let the assumptions of Theorem \ref{asym} hold and $u$ be a
positive solution to \eqref{cauchy}. There exist no strictly
positive constants $C, \beta$ such that $u(r)\leq C
(\psi(r))^{-\beta}$ for all $r\geq 0$.
\end{lem}

\begin{proof}
Assume by contradiction that there exist $C,\beta$ such that
$u(r)\leq C (\psi(r))^{-\beta}$ for all $r\geq 0$. It is not
restrictive assuming that $\beta<(n-1)/p$.

After integration in $(0,r)$ we get
\begin{equation*}
u'(r)\ge -C^p (\psi(r))^{1-n}\int_0^r (\psi(s))^{n-1-\beta p} \,
ds  \qquad \text{for any } r>0 \, .
\end{equation*}
Integrating now in $(r,+\infty)$ we obtain
\begin{equation*}
u(r)\le C^p \int_r^{+\infty} \left( (\psi(s))^{1-n} \int_0^s
(\psi(t))^{n-1-\beta p} \, dt \right) ds \qquad \text{for any }
r>0 \, .
\end{equation*}
Then, by \eqref{eq:ex-lim} we have
\begin{equation*}
u(r)=O((\psi(r))^{-\beta p}) \qquad \text{as } r\to +\infty \, .
\end{equation*}
Iterating this procedure as in the proof of \cite[Lemma 5.2]{BGGV}
we deduce that for any $\eps>0$ there exists $C_\eps>0$ such that
\begin{equation} \label{eq:dec-M}
u(r)\le C_\eps (\psi(r))^{-(n-1-\eps)} \qquad \text{for any } r>0
\, .
\end{equation}
The next purpose is to obtain a lower bound on $u$ in order to
reach a contradiction with \eqref{eq:dec-M}.

Let now $P=P(r)$ be the function defined in the proof of Theorem
\ref{teo}. Since we are assuming $A=A(r)$ convex, by the proof of
Proposition \ref{convexity-A(r)} we deduce that $P$ is negative
and nonincreasing in $(0,+\infty)$.

Therefore
$$
\frac{\int_0^r (\psi(s))^{n-1} \, ds}{(\psi(r))^{n-1}}
\left(\frac{(u'(r))^2}2+\frac{(u(r))^{p+1}}{p+1}\right)+\frac{u(r)u'(r)}{p+1}<0
\qquad \text{for any } r>0 \, .
$$
In particular we obtain
\begin{equation} \label{eq:stima>0}
u'(r)+ \frac{2(\psi(r))^{n-1}}{(p+1)\int_0^r (\psi(s))^{n-1} \,
ds}\, u(r)>0 \qquad \text{for any } r>0 \, .
\end{equation}
By \eqref{lim} we deduce that
$$
\frac{(\psi(r))^{n-1}}{\int_0^r (\psi(s))^{n-1} \, ds}\sim
(n-1)\frac{\psi'(r)}{\psi(r)} \qquad \text{as } r\to +\infty
$$
and hence for any $\eps>0$ there exists $r_\eps>0$ such that
\begin{equation*}
u'(r)+ \frac{2(n-1+\eps)}{p+1}\frac{\psi'(r)}{\psi(r)} \, u(r)>0
\qquad \text{for any } r>r_\eps
\end{equation*}
and after integration it follows that there exists $\overline C>0$
such that
\begin{equation*}
u(r)>\overline C(\psi(r))^{-\frac{2(n-1+\eps)}{p+1}} \qquad
\text{for any } r>r_\eps \, .
\end{equation*}
Since $p>1$, this contradicts \eqref{eq:dec-M} and completes the
proof of the lemma.
\end{proof}
Next we prove

\begin{lem} \label{l:u'/u}
 Let the assumptions of Theorem \ref{asym} hold and $u$ be a
positive solution to \eqref{cauchy}. Then
\begin{equation*}
\frac{u'(r)}{u(r)}=o\left(\frac{\psi'(r)}{\psi(r)}\right) \qquad
\text{as } r\to +\infty \, .
\end{equation*}
\end{lem}

\begin{proof} As a first step, we may exclude the case in which
\begin{equation} \label{eq:limite}
\limsup_{r\to +\infty}
\frac{u'(r)}{u(r)}\frac{\psi(r)}{\psi'(r)}<0
\end{equation}
since otherwise we would have
$$
\frac{u'(r)}{u(r)}< -C_1\frac{\psi'(r)}{\psi(r)} \qquad \text{for
any } r>\overline r
$$
for some $C_1>0$ and $\overline r>0$, and after integration it
follows
$$
u(r)<C_2 (\psi(r))^{-C_1} \qquad \text{for any } r>\overline r
$$
for some constant $C_2>0$, in contradiction with Lemma \ref{no exp
dec}. Now it is sufficient to prove existence of the limit in
\eqref{eq:limite}.

Suppose by contradiction that such a limit does not exist. For
simplicity, here we consider only the case $l=+\infty$ since the
case $l$ finite can be treated exactly as in \cite[Lemma
5.3]{BGGV}. Let $r_m\to +\infty$ be the sequence of local maxima
and minima points for
$\frac{u'(r)}{u(r)}\frac{\psi(r)}{\psi'(r)}$. Then for any $m$ we
have
$$
u''(r_m)u(r_m)-(u'(r_m))^2=u(r_m)u'(r_m)\left[\log\left(\frac{\psi'(r_m)}{\psi(r_m)}
\right) \right]'
$$
By \eqref{cauchy}, \eqref{lim}, \eqref{eq:stima>0} and $p>1$, it
follows that
\begin{align*}
&
u'(r_m)>-\frac{(u(r_m))^{p+1}}{u(r_m)\left\{(n-1)\frac{\psi'(r_m)}{\psi(r_m)}
-\frac{2}{p+1}\frac{(\psi(r_m))^{n-1}}{\int_0^{r_m}
(\psi(s))^{n-1}ds}+\left[\log\left(\frac{\psi'(r_m)}{\psi(r_m)}
\right) \right]' \right\}} \, .
\end{align*}
By \eqref{lim}, the fact that $l=+\infty$ and that
$\left[\log\left(\frac{\psi'(r_m)}{\psi(r_m)} \right) \right]'$ is
bounded we obtain
\begin{align*}
& u'(r_m)>-\frac{(u(r_m))^{p}}{\frac{(n-1)(p-1)}{p+1}
\frac{\psi'(r_m)}{\psi(r_m)}+o\left(\frac{\psi'(r_m)}{\psi(r_m)}
\right)} \, .
\end{align*}
This shows that
$\frac{u'(r_m)}{u(r_m)}\frac{\psi(r_m)}{\psi'(r_m)}\to 0$ as $m\to
+\infty$ and by the definition of $\{r_m\}$ we infer
$$
\lim_{r\to +\infty} \frac{u'(r)}{u(r)}\frac{\psi(r)}{\psi'(r)}=0
\, ,
$$
a contradiction. This completes the proof of the lemma.
\end{proof}

\begin{lem} Let the assumptions of Theorem \ref{asym} hold and $u$ be a
positive solution to \eqref{cauchy}. Then
\begin{equation} \label{eq:lim-u'-up}
\lim_{r\to +\infty}
\frac{u'(r)}{u^p(r)}\frac{\psi'(r)}{\psi(r)}=-\frac{1}{n-1} \, .
\end{equation}
\end{lem}

\begin{proof} We omit the proof in the case $l$ finite since it is
completely similar to the proof obtained in \cite[Section
5]{BGGV}. Let $l=+\infty$. First we
prove the existence of the limit in \eqref{eq:lim-u'-up}. Suppose
by contradiction that the limit in \eqref{eq:lim-u'-up} does not
exist. Then there exists a sequence $r_m\to +\infty$ of local
maxima and minima points for the function
$\frac{u'(r)}{u^p(r)}\frac{\psi'(r)}{\psi(r)}$. Then we have
\begin{equation*}
u''(r_m)u(r_m)=p(u'(r_m))^2-u(r_m)u'(r_m)\left[\log\left(\frac{\psi'(r_m)}{\psi(r_m)}\right)\right]'
\, .
\end{equation*}
Inserting this identity in \eqref{cauchy} multiplied by $u$ we
obtain
\begin{equation*}
u'(r_m)=-\frac{u^{p}(r_m)}{(n-1)\frac{\psi'(r_m)}{\psi(r_m)}+p\frac{u'(r_m)}{u(r_m)}-
\left[\log\left(\frac{\psi'(r_m)}{\psi(r_m)}\right)\right]'} \, .
\end{equation*}
Therefore we have
\begin{align*}
\frac{u'(r_m)}{u^p(r_m)}\frac{\psi'(r_m)}{\psi(r_m)} =
-\left\{n-1+p\frac{u'(r_m)}{u(r_m)}\frac{\psi(r_m)}{\psi'(r_m)}-\frac{\psi(r_m)}{\psi'(r_m)}
\left[\log\left(\frac{\psi'(r_m)}{\psi(r_m)}\right)\right]'\right\}^{-1}
\, .
\end{align*}
By Lemma \ref{l:u'/u}, the fact that $l=+\infty$ and that
\eqref{eq:supplementary} holds true, we obtain
\begin{align*}
\lim_{m\to +\infty}
\frac{u'(r_m)}{u^p(r_m)}\frac{\psi'(r_m)}{\psi(r_m)}=-\frac{1}{n-1}
\, .
\end{align*}
By definition of the sequence $\{r_m\}$, this gives the existence
of the limit in \eqref{eq:lim-u'-up}, a contradiction. It remains
to compute explicitly the limit in \eqref{eq:lim-u'-up}.

By \eqref{cauchy} we obtain
\begin{equation*}
\frac{u''(r)}{u'(r)}\frac{\psi(r)}{\psi'(r)}+n-1+\frac{u^p(r)}{u'(r)}\frac{\psi(r)}{\psi'(r)}=0
\end{equation*}
and hence there exists the limit
\begin{equation} \label{eq:lim-(1-n)}
\lim_{r\to +\infty}
\frac{u''(r)}{u'(r)}\frac{\psi(r)}{\psi'(r)}=1-n-\lim_{r\to
+\infty} \frac{u^p(r)}{u'(r)}\frac{\psi(r)}{\psi'(r)}\,.
\end{equation}
On the other hand, by de
l'H\^{o}pital rule and Lemma \ref{l:u'/u} we have
\begin{align*}
0=\lim_{r\to +\infty}
\frac{u'(r)}{u(r)}\frac{\psi(r)}{\psi'(r)}=\lim_{r\to +\infty}
\frac{[u'(r)\psi(r)(\psi'(r))^{-1}]'}{u'(r)}\\
=\lim_{r\to +\infty} \frac{\psi(r)}{\psi'(r)}
\left\{\frac{u''(r)}{u'(r)}-\left[\log\left(\frac{\psi'(r)}{\psi(r)}\right)\right]'\right\}
=\lim_{r\to +\infty} \frac{\psi(r)}{\psi'(r)}\frac{u''(r)}{u'(r)}
\, .
\end{align*}
Combining this with \eqref{eq:lim-(1-n)} we arrive to the
conclusion of the proof.
\end{proof}

\noindent\textit{End of the proof of Theorem \ref{asym}.} Using
\eqref{eq:lim-u'-up} we have that for any $\eps>0$ there exists
$r_\eps>0$ such that
\begin{equation*}
u^{1-p}(r_\eps)+\left(\frac{p-1}{n-1}-\eps\right)\int_{r_\eps}^r
\frac{\psi(s)}{\psi'(s)}\, ds<u^{1-p}(r)<u^{1-p}(r_\eps)+
\left(\frac{p-1}{n-1}+\eps\right)\int_{r_\eps}^r
\frac{\psi(s)}{\psi'(s)} \, ds \, .
\end{equation*}
If the function $\frac{\psi}{\psi'}$ is integrable in a
neighborhood of infinity then $\lim_{r\to +\infty} u(r)>0$. If
$\frac{\psi}{\psi'}$ is not integrable in a neighborhood of
infinity then $u$ vanishes at infinity and
\begin{equation*}
\lim_{r\to +\infty} \left(\int_{r_\eps}^r
\frac{\psi(s)}{\psi'(s)}\, ds\right)^{1/(p-1)}
u(r)=\left(\frac{n-1}{p-1}\right)^{1/(p-1)} \, .
\end{equation*}
This completes the proof of the theorem.

\section{Proof of the results in the subcritical case}

\subsection{Proof of Theorem \ref{exist gr st}}
By standard arguments we deduce that the bottom of the $L^2$ spectrum of
$-\Delta_g$ in $M$ admits the following variational
characterization:
\begin{equation}\label{lambda1}
\lambda_1(M):= \inf_{\varphi \in C^{\infty}_{c}(M)\setminus
\{0\}}\frac{\int_{M} |\nabla_g \varphi |_{g}^2\,dV_{g}}{\int_{M}
\varphi^2\,dV_{g}}\,.
\end{equation}
We start by proving the positivity of $\lambda_1(M)$, and by observing that if instead $\psi^\prime/\psi$ tends to zero, such positivity is false.

\begin{lem} \label{eq:lambda_1>0}
Let $n\ge 3$ and assume that $\psi$ satisfies assumptions
$(H_1)-(H_3)$. Then $\lambda_1(M)>0$. If instead $(H_3)$ does not hold and one has in addition $\psi^\prime(r)/\psi(r)\to0$ as $r\to+\infty$, then $\lambda_1(M)=0$. In particular, the latter fact holds if the radial mean curvature at $r$, or the radial sectional curvature at $r$, tend to zero as $r\to+\infty$.
\end{lem}

\begin{proof} Let $\lambda_1(B_R)$ be the infimum of the
functional in \eq{lambda1} with test functions in
$C^{\infty}_{c}(B_R)$, namely $\lambda_1(B_R)$ is the first
eigenvalue of the Laplace-Beltrami operator on $B_R$ under the
Dirichlet boundary condition. From \cite{Gr} we recall the
estimate
\begin{equation}\label{estimate}
\lambda_1(B_R)\geq \frac{1}{4 F(R)}\,,
\end{equation}
where $F(R):=\displaystyle{ \sup_{0<r<R} H_R(r)}$ for any $R\in
(0,+\infty)$ and
$$
\displaystyle{H_R(r):=\left[\left(\int_0^r (\psi(s))^{n-1}\,ds
\right) \left(\int_{r}^{R}(\psi(s))^{1-n}\,ds\right) \right]} \, .
$$
Since the map $R\mapsto \lambda_1(B_R)$ is decreasing and
$\lambda_1(M)=\displaystyle{\lim_{R\rightarrow +\infty}
\lambda_1(B_R)}$, one has
$$\lambda_1(M)\geq \displaystyle{\lim_{R\rightarrow +\infty} \frac{1}{4 F(R)}}\,.$$
In particular, the claim can be proved by showing that $F(R)$
stays bounded.\par
 We have that $\displaystyle{\lim_{r\rightarrow R^-} H_R(r)=0}$ and, by applying twice
 the l'H\^{o}pital rule, that
$$\lim_{r\rightarrow 0^+} H_R(r)=\lim_{r\rightarrow 0^+} \left(\frac{\int_0^r (\psi(s))^{n-1}\,ds}{(\psi(r))^{n-1}}\right)^2
= \lim_{r\rightarrow 0^+} \left(\frac{
\psi(r)}{(n-1)\psi'(r)}\right)^2=0\,.$$ On the other hand, for
$0<r<R$, we have
$$
H'_R(r)=(\psi(r))^{n-1}\left(\int_{r}^{R}(\psi(s))^{1-n}\,ds\right)-(\psi(r))^{1-n}\left(\int_0^r
(\psi(s))^{n-1}\,ds \right) \, .
$$
Since $\displaystyle{\lim_{r\rightarrow 0^+} H_R(r)=\lim_{r\to
R^-} H_R(r)=0}$ and $H_R(r)>0$ for any $r\in (0,R)$, then $H_R$
admits a local maximum point $r_0\in (0,R)$. This yields,
$$
H_R(r)\leq H_R(r_0)=\left(
\frac{\int_{0}^{r_0}(\psi(s))^{n-1}\,ds}{(\psi(r_0))^{n-1}}\right)^2\quad
\text{for every }
 r\in(0,R)\,.
$$
Then, condition $(H_3)$ assures the boundedness of the latter
quotient and, in turn, proves the claim. To see this, note that
$(H_2)-(H_3)$ yield $\displaystyle{\lim_{r\rightarrow +\infty}
\psi(r)=+\infty}$. In particular, by the Cauchy Theorem
$$
\limsup_{r\rightarrow
+\infty}\frac{\int_{0}^{r}(\psi(s))^{n-1}\,ds}{(\psi(r))^{n-1}}\leq
\limsup_{r\rightarrow +\infty}
\frac{\psi(r)}{(n-1)\psi'(r)}<+\infty\,.
$$
To prove the second part of the statement, we notice that, denoting by $V(r)$ the volume of the geodesic balls of radius $r$ centered at $o$, we have by l'H\^{o}pital rule, since the last limit below exists:
\[
\lim_{r\to+\infty}\frac{V^\prime(r)}{V(r)}=\lim_{r\to+\infty}\frac{(\psi(r))^{n-1}}{\int_0^r(\psi(s))^{n-1}\,
ds}=(n-1)\lim_{r\to+\infty}\frac{\psi^\prime(r)}{\psi(r)}=0.
\]
Hence for all $\varepsilon>0$ there is $r_\varepsilon$ such that $0\le V^\prime(r)/V(r)
=\left(\log V(r)\right)^\prime\le \varepsilon$ for all $r\ge r_\varepsilon$. Integrating between $r_\varepsilon$ and $r$ we
easily get that
\[
\lim_{r\to+\infty}\frac{\log(\psi(r))}r=0.
\]
By a classical result of Brooks (see \cite{Brooks}) this implies
that $\lambda_1(M)=0$. We mention by the sake of completeness that
the same conclusion can be obtained by verifying that the
necessary and sufficient condition \eqref{Sob rad}  for the
validity of the Poincar\'e-Sobolev type inequality below, is not
satisfied when $p=1$ under the running assumptions.

The fact that the claim holds when the radial mean curvature at $r$ tends to zero when $r\to+\infty$ is obvious from its expression given in Section \ref{notation}. The radial sectional curvature at $r$ tends to zero when $r\to+\infty$ if and only if $\psi^{\prime\prime}(r)/\psi(r)\to0$. This implies that $\psi^{\prime}(r)/\psi(r)\to0$ as well. In fact, if $\psi^{\prime}(r)/\psi(r)$ has a limit, l'H\^{o}pital rule implies that it must be zero, as claimed. Should $\psi^{\prime}(r)/\psi(r)$ not have a limit, it must have a sequence of stationary points $r_k\to+\infty$ as $k\to+\infty$, so that $\psi^{\prime\prime}(r_k)/\psi^\prime(r_k)=\psi^{\prime}(r_k)/\psi(r_k)$, so that in particular $\psi^{\prime\prime}(r_k)\not=0$. Hence
\[\frac{\psi^\prime(r_k)}{\psi(r_k)}=\frac{\psi^{\prime\prime}(r_k)}{\psi(r_k)}\,\frac{\psi^{\prime}(r_k)}{\psi^{\prime\prime}(r_k)}=\frac{\psi^{\prime\prime}(r_k)}
{\psi(r_k)}\,\frac{\psi(r_k)}{\psi^\prime(r_k)},\ \ {\rm or} \
\left(\frac{\psi^\prime(r_k)}{\psi(r_k)}\right)^2
=\frac{\psi^{\prime\prime}(r_k)}{\psi(r_k)}\] which tends to zero
by assumption if $k\to+\infty$, contradiction.
\end{proof}

Next we show the validity of a Sobolev embedding for the space
$H^1_{r}(M)$ of radial functions in $H^1(M)$.

\begin{lem}\label{Sob Kufner}
Let $n\ge 3$ and assume that $\psi$ satisfies $(H_1)-(H_3)$. If $1<p\le \frac{n+2}{n-2}$ then the embedding
$H^1_{r}(M)\subset L^{p+1}(M)$ is continuous and if
$1<p<\frac{n+2}{n-2}$ then the embedding is also compact.
\end{lem}

\begin{proof}
Following \cite{kufner}, we define $AC_{R}(0,+\infty)$ the set of
all functions absolutely continuous on every compact subinterval
$[a,b]\subset (0,+\infty)$ which tend to zero as $r\to +\infty$.
Then, according to \cite[Theorem 6.2]{kufner}, the inequality
\begin{equation}\label{Sob rad}
\left(\int_{0}^{+\infty}|u(r)|^{p+1}\psi^{n-1}(r)\,dr\right)^{\frac{2}{p+1}}\leq
C_{n,p}\int_{0}^{+\infty} (u'(r))^2\psi^{n-1}(r)\,dr \quad
\text{for all }u\in AC_{R}(0,+\infty)\,,
\end{equation}
holds for some $C_{n,p}>0$, if and only if
\[\label{ns}\sup_{0<x<+\infty}f_{n,p}(x):=\sup_{0<x<+\infty}\left( \int_0^{x} \psi^{n-1}(r)\,dr\right)^{\frac{1}{p+1}}\left( \int_x^{+\infty} \psi^{1-n}(r)\,dr\right)^{\frac{1}{2}}<+\infty\,.\]
The known asymptotics for $\psi$ as $x\rightarrow 0^+$ yield
\begin{equation} \label{eq:asymp}
f_{n,p}(x)\sim \frac{1}{n^{1/(p+1)} (n-2)^{1/2}} \,
x^{\frac{n+2-p(n-2)}{2(p+1)}} \qquad \text{as } x\rightarrow 0^+,
\end{equation}
where the integrability of $\psi^{1-n}(r)$ in $(x,+\infty)$ comes
from $(H_3)$.\par On the other hand, we claim that
\begin{equation}\label{claim}
\limsup_{x\rightarrow +\infty}\left( \int_0^{x}
\psi^{n-1}(r)\,dr\right)\left( \int_x^{+\infty}
\psi^{1-n}(r)\,dr\right)<+\infty
\end{equation}
from which we easily conclude that for $p>1$
$$\lim_{x \rightarrow +\infty}f_{n,p}(x)= \lim_{x \rightarrow +\infty}\left( \int_0^{x} \psi^{n-1}(r)\,dr\right)^{\frac{1}{p+1}-\frac{1}{2}}\left( \int_0^{x} \psi^{n-1}(r)\,dr\right)^{\frac{1}{2}}\left( \int_x^{+\infty} \psi^{1-n}(r)\,dr\right)^{\frac{1}{2}}  =0\,.$$
To prove \eq{claim}, we first note that $(H_2)-(H_3)$ and the
Cauchy's Theorem yield
$$
\limsup_{x\rightarrow +\infty}\frac{ \int_0^{x}
\psi^{n-1}(r)\,dr}{(\psi(x))^{n-1}}\leq \limsup_{x\rightarrow
+\infty} \frac{\psi(x)}{(n-1)\psi'(x)}<+\infty\,.$$ and
$$
\limsup_{x\rightarrow +\infty}\frac{\int_x^{+\infty}
\psi^{1-n}(r)\,dr}{(\psi(x))^{1-n}}\leq \limsup_{x\rightarrow
+\infty} \frac{\psi(x)}{(n-1)\psi'(x)}<+\infty\,.
$$
Then,
$$
\limsup_{x\rightarrow +\infty} \left( \int_0^{x}
\psi^{n-1}(r)\,dr\right)\left( \int_x^{+\infty}
\psi^{1-n}(r)\,dr\right)$$
$$
= \limsup_{x\rightarrow +\infty}\left(\frac{ \int_0^{x}
\psi^{n-1}(r)\,dr}{(\psi(x))^{n-1}}\right)
\left(\frac{\int_x^{+\infty}
\psi^{1-n}(r)\,dr}{(\psi(x))^{1-n}}\right)<+\infty\,.
$$
Let us denote by $C^{\infty}_{c,r}(M)$ the space of radial
functions in $C^{\infty}_{c}(M)$. By \eqref{Sob
rad},\eqref{eq:asymp}, \eqref{claim} we deduce that if $1<p\le
\frac{n+2}{n-2}$ then
$$
\|\varphi\|_{L^{p+1}(M)}^2 \le C_{n,p} \int_M |\nabla_g
\varphi|_g^2 \, dV_g \, ,
$$
for any function $\varphi\in C^\infty_{c,r}(M)$.

Therefore, by density of $C^{\infty}_{c,r}(M)$ in $H^1_{r}(M)$
(see \cite[Theorem 3.1]{hebey}) we obtain the continuous embedding
$H^1_{r}(M)\subset L^{p+1}(M)$ for $1<p\leq \frac{n+2}{n-2}$. On
the other hand \cite[Theorem 7.4]{kufner} yields that the same
embedding is compact if and only if
$\displaystyle{\lim_{x\rightarrow 0^+}f_{n,p}(x)=0}$
$\displaystyle{= \lim_{x \rightarrow +\infty}f_{n,p}(x)}$. This
condition is satisfied when $1<p<\frac{n+2}{n-2}$.
\end{proof}

\textit{End of the proof of Theorem \ref{exist gr st}.}  The
existence of a nonnegative minimizer to
\begin{equation} \label{eq:Sob-Sub}
\inf_{v\in H^1_{r}(M)\setminus \{0\}} \frac{\int_{M} |\nabla_g
v|_g^2 \, dV_g}{\left(\int_{M} |v|^{p+1} \,
dV_g\right)^{\frac{2}{p+1}}} \, ,
\end{equation}
follows in a standard way by Lemmas \ref{eq:lambda_1>0}-\ref{Sob
Kufner}. Up to a constant multiplier a nonnegative minimizer $u$
of \eqref{eq:Sob-Sub} is actually a radial solution of \eqref{LEF}
 and hence a nonnegative solution of \eqref{cauchy}. Furthermore
 $u(r)>0$ for any $r>0$ by local uniqueness for a Cauchy
 problem.

\subsection{Proof of Theorem \ref{subcr set}}
The proof follows the line of \cite[Theorem 1.3]{mancini} where
the case $\psi(r)=\sinh(r)$ is dealt. Hence, in the sequel we will
only quote which are the main differences.\par First we have
uniqueness for Dirichlet problems on bounded domains.
\begin{lem}\label{unique dir}
Let $1<p<\frac{n+2}{n-2}$ and $\psi$ satisfy assumptions
$(H_1)-(H_2)$. Furthermore, let $G$ as defined in Theorem
\ref{subcr set} satisfying the $\Lambda-$property as required
there. Then the Dirichlet problem
\begin{equation}\label{Dir}
\left \{ \begin{array}{ll}
-\displaystyle{\frac1{(\psi(r))^{n-1}}\left[(\psi(r))^{n-1} v^\prime(r)\right]^\prime}=|v(r)|^{p-1}v(r) & r\in (0,R)\\ \\
v'(0)=0\quad v(R)=0\,. &
\end{array}
\right.
\end{equation}
 has at most one positive solution.
\end{lem}

\begin{proof}
The proof follows plainly the lines of \cite[Proposition
4.4]{mancini}. The main difference is that the auxiliary energy
considered there, here has to be replaced by
$$E_{\widehat{v}}(r):=\frac{1}{2} (\psi(r))^{\delta(p-1)}(\widehat{v}'(r))^2+\frac{|\widehat{v}(r)|^{p+1}}{p+1}+\frac{1}{2}\,G(r)(\widehat{v}(r))^2\,,$$
where $\delta$ and $G$ are as in the statement of Theorem
\ref{subcr set} and $\widehat{v}(r):=\psi^{\delta}(r)v(r)$. See
also \cite{kwong} where this substitution was originally
introduced. In particular, if $v$ solves \eqref{Dir} then
$\widehat{v}$ solves
$$\psi^{\delta(p-1)}(r)\,\widehat{v} ''(r)+\frac{1}{2}(\psi^{\delta(p-1)}(r))'\,\widehat{v}'(r)+G(r)\,\widehat{v}(r)+\widehat{v}^p(r)=0\quad (r>0)$$
and
$$\frac{d}{dr}\,E_{\widehat{v}}(r)=\frac{1}{2}\,G'(r)(\widehat{v}(r))^2\,.$$
We have $G(r)\sim \delta(\delta+2-n)r^{\delta(p-1)-2}$ as
$r\rightarrow 0^+$, where $\delta+2-n<0$ and $\delta(p-1)-2<0$.
Namely, $G(r)\rightarrow -\infty$ for $r\rightarrow 0^+$. This,
combined with the assumptions required on $G$ yields that either
$G'(r)\geq 0$ for every $r>0$ or there exists $r_1>0$ such that
$G'(r_1)=0$, $G'(r)\geq 0$ for every $r\in (0,r_1)$ and $G'(r)\leq
0$ for every $r>r_1$. Then, all the arguments of \cite[Proposition
4.4]{mancini} work. See also the proof of Lemma \ref{uniqueness}
below.
\end{proof}

Let $u\in H^1(M)$ be a positive radial solution of \eqref{LEF} as
given in Theorem \ref{exist gr st} (possibly not unique). The next
two lemmas show that every solution $v$ to \eqref{cauchy} with
$0<v(0)<u(0)$, is necessarily of one sign. Furthermore, $v$
intersects $u$ exactly once. First, by exploiting Lemma
\ref{unique dir}, we have
\begin{lem}\label{G}
Let $1<p<\frac{n+2}{n-2}$ and $\psi$ satisfy assumptions
$(H_1)-(H_3)$. Furthermore, let $G$ as defined in Theorem
\ref{subcr set} satisfying the $\Lambda-$property as required
there. If $u$ and $v$ are two solutions to \eqref{cauchy} with
$u(r)>0$ for every $r\geq 0$ and $0<v(0)<u(0)$, then $v(r)>0$ for
every $r\geq 0$.
\end{lem}
The proof of Lemma \ref{G} is the same of \cite[Lemma 4.1 and
Corollary 4.6]{mancini}. The main tools exploited there are
uniqueness for Dirichlet problems on bounded domains and the
Poincar\'e-Sobolev inequality in the hyperbolic space. In our
case, they are given, respectively, by Lemma \ref{unique dir} and
by Lemmas \ref{eq:lambda_1>0}-\ref{Sob Kufner}.

On the other hand, exactly as in \cite[Corollary 4.6]{mancini},
one shows
\begin{lem}\label{Uv}
Let $1<p<\frac{n+2}{n-2}$ and $\psi$ satisfy assumptions
$(H_1)-(H_3)$. Let $u$ and $v$ be two positive solutions to
\eqref{cauchy} with $0<v(0)<u(0)$. If $u\in H^1(M)$, then $u-v$
has exactly one zero.
\end{lem}

Next we discuss the asymptotic behavior and uniqueness of radial
ground states.

\begin{lem}\label{asym U}
Assume that $\psi$ satisfies assumptions $(H_1)-(H_2)$ and
$\eq{l}$ holds. Furthermore, let $u\in H^1(M)$ be a positive
solution to \eqref{cauchy} with $p>1$. If $l<+\infty$ in $\eq{l}$,
then
\begin{equation}\label{lim1}
\lim_{r\rightarrow +\infty}\frac{\log
(u(r))}{r}=-(n-1)l=\lim_{r\rightarrow +\infty}\frac{\log
|u'(r)|}{r}
\end{equation}
and
\begin{equation}\label{lim2}
\lim_{r\rightarrow +\infty}\frac{u'(r)}{u(r)}=-(n-1)l\,.
\end{equation}
If $l=+\infty$ in $\eq{l}$, we have
\begin{equation}\label{lim3}
\lim_{r\rightarrow +\infty}\frac{\log
(u(r))}{r}=\lim_{r\rightarrow +\infty}\frac{\log |u'(r)|}{r}=
-\infty\, , \qquad \lim_{r\rightarrow
+\infty}\frac{u'(r)}{u(r)}=-\infty \, ,
\end{equation}
and
\begin{equation} \label{lim4}
\lim_{r\rightarrow
+\infty}\frac{\log|u'(r)|}{\log(\psi(r))}=-(n-1) \,.
\end{equation}
\end{lem}

\medskip\par

\begin{proof} We omit the proof in the case $l<+\infty$ since it
can be deduced by arguing as in \cite[Lemma 3.4]{mancini}.
Suppose now that $l=+\infty$. For every $k>0$ there exists
$r_{k}>0$ such that
$$u''(r)+(n-1)k u'(r)+\,\frac{1}{k} u(r)\geq 0 \quad \text{for all } r\geq r_{k}\,. $$
Namely,
$$\left( e^{-\lambda_{-}(k) r}z(r)\right)'\geq 0 \quad \text{for all } r\geq r_{k}\,,$$
where $z:=u'-\lambda_{+}(k)u$ and
$\lambda_{\pm}(k):=\frac{-(n-1)k\pm\sqrt{(n-1)^2k^2-4/k}}{2}$.
Then, two integrations in $[\tau,r]$, with $r_{k}\leq \tau \leq
r$, yield
$$u(r)\geq B_k(\tau)\, e^{\lambda_{+}(k)r}-\frac{A_k(\tau)}{\lambda_{+}(k)-\lambda_{-}(k)}\,e^{\lambda_{-}(k)r}\quad \text{for all } r\geq r_{k}\,,$$
where $A_k(\tau):=e^{-\lambda_{-}(k)\tau}z(\tau)$ and
$B_k(\tau):=u(\tau)e^{-\lambda_{+}(k)\tau}+\frac{A_k(\tau)}{\lambda_{+}(k)-\lambda_{-}(k)}\,e^{-(\lambda_{+}(k)-\lambda_{-}(k))\tau}$.
We claim that $B_k(\tau) \leq 0$ for $\tau\geq r_k$. Otherwise,
$B_k(\tau)>0$ eventually.  We recall that
$$
B'_k(\tau)=\frac{A'_k(\tau)}{\lambda_{+}(k)-\lambda_{-}(k)}\,e^{-(\lambda_{+}(k)-\lambda_{-}(k))\tau}\geq
0 \qquad \text{for any } \tau\geq r_{k} \, .
$$
Here and in the sequel $C_k$ denotes a positive constant
sufficiently large which may vary from line to line.  Then,
$u(r)\geq  B_k(\tau) \,
e^{\lambda_{+}(k)r}+o(e^{\lambda_{+}(k)r})$ as $r\to +\infty$. But
this, combined with $\eq{l}$, yields $\int_0^{+\infty}
\psi^{n-1}(r)u^2(r)\,dr \geq C_k\int_{r_k}^{+\infty}
e^{\sqrt{(n-1)^2k^2-4/k}\,r}\,dr$ for some $C_k>0$ and contradicts
the fact that $u\in H^1(M)$. Hence, $B_k(\tau) \leq 0$ for
$\tau\geq r_k$ and we conclude that
\begin{equation}\label{bound}
u'(\tau)\leq \lambda_{-}(k) u(\tau) \quad \text{for all } \tau
\geq r_{k}\,.
\end{equation}
Then,
$$
\limsup_{r\rightarrow +\infty} \frac{u'(r)}{u(r)}\le \lambda_-(k),
\qquad \limsup_{r\rightarrow +\infty} \frac{\log (u(r))}{r}\leq
\lambda_{-}(k)\quad \text{for every } k >0$$
 and the first and the third limit in \eqref{lim3} follow since $\displaystyle{\lim_{k \rightarrow +\infty}
\lambda_{-}(k)=-\infty}$. On the other hand, by \eqref{cauchy}, the
third limit in \eqref{lim3}, the fact that $\displaystyle{\lim_{r\to +\infty}
u(r)=0}$ and that $l=+\infty$, we have
$$
\lim_{r\to +\infty}
\frac{u''(r)}{u'(r)}\frac{\psi(r)}{\psi'(r)}=\lim_{r\to +\infty}
\left(-(n-1)-\frac{u^p(r)}{u'(r)}\frac{\psi(r)}{\psi'(r)}\right)=1-n\,.
$$
By this, the second limit in \eqref{lim3} and \eqref{lim4} easily
follow from the l'H\^{o}pital rule.
\end{proof}

\begin{lem} \label{l:asym u'}
Let $n\ge 3$ and $1<p<\frac{n+2}{n-2}$. Assume that
$\psi$ satisfies $(H_1)-(H_3)$. Then for any radial positive
solution $u\in H^1(M)$ of \eqref{LEF}, there exists $L\in
(-\infty,0)$ such that
\begin{equation}\label{asim u'}
\lim_{r\rightarrow +\infty}\psi^{n-1}(r) u'(r)=L \, .
\end{equation}
Moreover
\begin{equation}\label{asim u1} \lim_{r\rightarrow
+\infty}\psi^{n-1}(r) u(r)=\frac{|L|}{(n-1)l} \qquad \text{if }
l<+\infty \, ,
\end{equation}
and
\begin{equation}\label{asim u2}
\lim_{r\rightarrow
+\infty}\frac{u(r)}{\int_{r}^{+\infty}\psi^{1-n}(s)\,ds}=|L|
\qquad \text{if } l=+\infty \, .
\end{equation}
\end{lem}

\begin{proof} The existence and the negativity of the limit in \eqref{asim u'} simply
follows by \eqref{cauchy}. It remains to prove that $L>-\infty$.
If $l<+\infty$ from \eqref{lim1} the bound
$$
u(r)\leq C_{\delta}e^{-((n-1)l-\delta) r} \quad \text{for all
}r\geq 0\,,
$$
holds for every $\delta>0$. By this, \eqref{l} and \eqref{cauchy},
we deduce that
$$(\psi^{n-1}(r)u'(r))'\geq -C(\varepsilon, \delta)e^{-[((n-1)l-\delta)p-(l+\varepsilon)(n-1)] r} \quad \text{for all }r\geq 0\,,$$
for every $\varepsilon>0$ and $\delta>0$, where $C(\varepsilon,
\delta)>0$. Next we fix $\varepsilon=\frac{l(p-1)}{2}$ and we
assume $\delta=\delta(\varepsilon)$ to be such that $\delta
p<\frac{l(p-1)(n-1)}{2}$. Then, an integration in $[0,r]$ yields
$$\psi^{n-1}(r)u'(r)\geq C[e^{-[\frac{l(p-1)(n-1)}{2}-\delta p]r}-1]\,.$$
Namely,
$$\psi^{n-1}(r)u'(r)\geq - C  \quad \text{for all }r\geq 0\,$$
and $L>-\infty$.\par Next we assume $l=+\infty$. From \eqref{lim3}
we know that for every $\varepsilon>0$ there exists
$R_{\varepsilon}>0$ such that
$$
\log|u'(r)|\leq -((n-1)-\varepsilon)\log \psi(r) \quad \text{for
all }r\geq R_{\varepsilon}\,.
$$
Furthermore, from \eqref{bound}, for every $k>0$ there exists
$r_k>0$ such that
$$\log u(r)\leq \log|u'(r)|-\log|\lambda_{-}(k)|\quad \text{for all }r\geq r_k\,,$$
where $\displaystyle{\lim_{k \rightarrow +\infty}
\lambda_{-}(k)=-\infty}$. Fix $\varepsilon=\frac{(p-1)(n-1)}{2p}$
in order to obtain after integration
$$
u(r)\leq C \psi^{-(n-1)(p+1)/2p}(r)\quad \text{for all }r\geq 0 \,
,
$$
for some $C>0$. By this and integrating the equation in
$[0,r]$, we conclude that
$$\psi^{n-1}(r)u'(r)\geq - C^p\int_{0}^{r}\psi^{-(n-1)(p-1)/2}(s)\,ds\geq -K\,,$$
for some finite $K$ and for all $r\geq 0$. Hence, again we infer
that $L>-\infty$.
\end{proof}

\begin{lem}\label{uniqueness}
Let $1<p<\frac{n+2}{n-2}$. Assume that $\psi$ satisfies the
assumptions of Theorem \ref{subcr set}.\par Then \eqref{LEF}
admits a unique radial positive solution $U\in H^1(M)$.
\end{lem}

\begin{proof}
We follow the proof of \cite[Theorem 1.3]{mancini}. By
contradiction, assume that $u$ and $v$ are two positive solutions
to \eqref{cauchy} such that $u,v\in H^1(M)$ and $v(0)<u(0)$. By
Lemma \ref{Uv}, $u$ and $v$ intersect exactly once at $r_0$.\par
We claim that $\gamma(r):=v(r)/u(r)$ is strictly increasing in
$(0,+\infty)$. From the equation we know that
$$[(\psi(r))^{n-1}(v'(r)u(r)-v(r)u'(r))]'=(\psi(r))^{n-1}u(r)v(r)((u(r))^{p-1}-(v(r))^{p-1})\,.$$
Hence,
$$
[(\psi(r))^{n-1}(v'(r)u(r)-v(r)u'(r))]'(r_0-r)>0 \quad \forall
r\neq r_0 \, .
$$
By \eq{asim u'} and the fact that $\displaystyle{\lim_{r\to +\infty}
u(r)=\lim_{r\to +\infty} v(r)=0}$,  we deduce that
\begin{equation}\label{v'u}
\lim_{r\rightarrow +\infty}
(\psi(r))^{n-1}(v'(r)u(r)-v(r)u'(r))=0\,.
\end{equation}
Hence, $v'(r)u(r)-v(r)u'(r)>0$ for $r>0$ and $\gamma'(r)>0$.\par
Now, we set $\widehat{u}(r):=(\psi(r))^{\delta}u(r)$ and
$\widehat{v}(r):=(\psi(r))^{\delta}v(r)$, where $\delta$ is as in
the statement of Theorem \ref{subcr set}. Then, for
$E_{\widehat{v}}$ as in the proof of Lemma \ref{unique dir}, for
any $0<\varepsilon<R$ and $r\in (0,R)$, we get
\begin{equation}\label{E}
E_{\widehat{v}}(R)-\gamma^2(r)E_{\widehat{u}}(R)=E_{\widehat{v}}(\varepsilon)-\gamma^2(r)E_{\widehat{u}}(\varepsilon)+\frac{1}{2}\int_{\varepsilon}^R
G'(s)[(\widehat{v}(s))^2-\gamma^2(r)(\widehat{u}(s))^2]\,ds\,.
\end{equation}
Since $G(r)\to -\infty$ as $r\to 0^+$ (see the proof of Lemma
\ref{unique dir}), by assumption we have that either $G'\geq 0$ in
$(0,+\infty)$ or there exists $r_1>0$ such that $G'(r_1)=0$,
$G'\geq 0$ in $(0,r_1)$ and $G'\leq 0$ in $(r_1,+\infty)$. We
claim that
\begin{equation}\label{R}
E_{\widehat{v}}(R)\rightarrow 0\quad \text{and} \quad
E_{\widehat{u}}(R)\rightarrow 0 \quad \text{as } R\rightarrow
+\infty\,.
\end{equation}
We now show that with the help of \eqref{R} we arrive to the
conclusion of the proof.

If $G'$ does not change sign, take $r=\varepsilon$ in \eq{E}.
Letting $\varepsilon \rightarrow 0^+$ we get
$$
E_{\widehat{v}}(R)-\gamma^2(0)E_{\widehat{u}}(R)=\frac{1}{2}\int_{0}^R
G'(s) [(\widehat{v}(s))^2-\gamma^2(0)(\widehat{u}(s))^2]\,ds>0 \,
.
$$
Letting $R\to +\infty$, \eqref{R} leads to a contradiction.

If $G'$ changes sign, take $r=r_1$ in \eq{E}. Letting $\varepsilon
\rightarrow 0^+$, we get
$$
E_{\widehat{v}}(R)-\gamma^2(r_1)E_{\widehat{u}}(R)
$$
$$
=\frac{1}{2}\int_{0}^{r_1}
G'(s)[(\widehat{v}(s))^2-\gamma^2(r_1)(\widehat{u}(s))^2]\,ds+\frac{1}{2}\int_{r_1}^R
G'(s)[(\widehat{v}(s))^2-\gamma^2(r_1)(\widehat{u}(s))^2]\,ds<0
\,.
$$
Letting $R\to +\infty$, \eq{R} leads again to a contradiction.

It remaind to prove \eq{R}. First we note that, from \eqref{asim
u1} and \eqref{asim u2}, if $l<+\infty$ we have
$$\psi^{\delta}(r)v(r)\sim \frac{|L|}{(n-1)l}\,\psi^{-\frac{(p+1)\delta}{2}}(r)\quad \text{as } r\rightarrow +\infty$$
and if $l=+\infty$ we have
$$\psi^{\delta}(r)v(r)\sim |L|\psi^{\delta}(r)\int_{r}^{+\infty} \psi^{1-n}(s)\,ds\quad \text{as } r\rightarrow +\infty\,.$$
Hence, in both the cases we conclude that
$\widehat{v}(r)\rightarrow 0$ as $r\rightarrow +\infty$. Then we
consider
$$G(r)(\widehat{v}(r))^2=\delta(\delta+2-n)\psi^{\delta(p+1)}(r)\left(\frac{\psi'(r)}{\psi(r)}\right)^2v^2(r)
-\delta\psi^{\delta(p+1)}(r)\frac{\psi'(r)}{\psi(r)}\frac{\psi''(r)}{\psi'(r)}v^2(r)\,.$$
If $l<+\infty$ \eqref{l} and \eqref{asim u1} give
$$G(r)(\widehat{v}(r))^2\sim \frac{\delta(\delta+1-n)|L|^2}{(n-1)^2}\,\psi^{-2\delta}(r)\quad \text{as } r\rightarrow +\infty$$
and $|G(r)|(\widehat{v}(r))^2\rightarrow 0$ as $r\rightarrow
+\infty$. If $l=+\infty$, \eqref{bound} and \eqref{asim u'} give
$$\psi^{\delta(p+1)}(r)\left(\frac{\psi'(r)}{\psi(r)}\right)^2v^2(r)\leq \left[\psi^{\frac{\delta(p+1)}{2}}(r)\frac{\psi'(r)}{\psi(r)}\frac{|v'(r)|}{|\lambda_{-}(k)|}\right]^2\sim \left[\frac{|L|\,\psi^{-\delta}(r)}{|\lambda_{-}(k)|}\frac{\psi'(r)}{\psi(r)}\right]^2\,,$$
as $r\rightarrow +\infty$. Hence, by \eqref{extra},
$\psi^{\delta(p+1)}(r)\left(\frac{\psi'(r)}{\psi(r)}\right)^2v^2(r)
\rightarrow 0$ as $r\rightarrow +\infty$. Similarly,\par
$\psi^{\delta(p+1)}(r)\frac{\psi'(r)}{\psi(r)}\frac{\psi''(r)}{\psi'(r)}$
$v^2(r)\rightarrow 0$ as $r\rightarrow +\infty$ and, in turn,
$|G(r)|(\widehat{v}(r))^2\rightarrow 0$ as $r\rightarrow
+\infty$.\par Finally, we compute
$$\psi^{\delta(p-1)}(r)(\widehat{v}'(r))^2$$
$$=\delta^2\psi^{\delta(p+1)}(r)\left(\frac{\psi'(r)}{\psi(r)}\right)^2v^2(r)
+2\delta
\psi^{\delta(p+1)}(r)\frac{\psi'(r)}{\psi(r)}v(r)v'(r)+\psi^{\delta(p+1)}(r)
(v'(r))^2\,.$$ If $l<+\infty$ \eqref{l}, \eqref{asim u'} and
\eqref{asim u1} give
$$
\psi^{\delta(p-1)}(r)(\widehat{v}'(r))^2\sim L^2\left(
\frac{\delta^2}{(n-1)^2}-\frac{2\delta}{n-1}+1\right)\,\psi^{-2\delta}(r)\quad
\text{as } r\rightarrow +\infty\,.
$$
Namely, $\psi^{\delta(p-1)}(r)(\widehat{v}'(r))^2 \rightarrow 0$
as $r\rightarrow +\infty$.
 When $l=+\infty$, the same conclusion can be reached by exploiting \eqref{extra}, \eqref{bound} and \eqref{asim u'} as shown above. The limits so far proved yield \eqref{R}.
\end{proof}

\medskip\par

When $\alpha$ is large, the same proof of \cite[Lemma 7.1]{BGGV}
gives
\begin{lem}\label{blow1}
Let $\psi$ satisfy assumptions $(H_1)-(H_2)$. Furthermore, let $u$
be a solution to \eqref{cauchy} with $1<p<\frac{n+2}{n-2}$ and
$\alpha> \alpha_0$ sufficiently large. Then $u$ changes sign.
\end{lem}

Finally, following the proofs of \cite[Lemmas 7.2, 7.3, 7.4,
7.5]{BGGV}, we conclude.

\begin{lem}\label{u(0)>U(0)}
Let $1<p<\frac{n+2}{n-2}$, $\psi$ satisfy the assumptions of
Theorem \ref{subcr set} and $U$ be the unique ground state as
given in Lemma \ref{uniqueness}. Then, any solution to
\eqref{cauchy} with $\alpha>U(0)$ is sign-changing.
\end{lem}

The proof of Theorem \ref{subcr set} now follows from Lemmas \ref{uniqueness}-\ref{u(0)>U(0)}.

\medskip

\subsection{Proof of Proposition \ref{remark2.5}.}
Note that $G'(r)=\delta \psi^{\delta(p-1)-3}(r)h(r)$, where
$$ h(r):=(\delta(p-1)-2)(\delta+2-n)(\psi'(r))^{3}-\psi'''(r)\psi^2(r)+(\delta(3-p)+5-2n)\psi'(r)\psi''(r)\psi(r)\,.$$
Clearly, $h(0)=(\delta(p-1)-2)(\delta+2-n)>0$ for every
$1<p<\frac{n+2}{n-2}$. We prove that $h'(\bar r)<0$ for every
$\bar r>0$ such that $h(\bar r)=0$, then $h$ admits at most one
zero and the $\Lambda-$property follows.

For such $\bar r$, a few computations yield
$$h'(\bar r)=A_{p,n}(\psi'(\bar r))^{2}\psi''(\bar r)+B_{p,n}\psi(\bar r)\psi'(\bar r)\psi'''(\bar r)+\psi^2(\bar r)\left(\frac{\psi''(\bar r)\psi'''(\bar r)}{\psi'(\bar r)}-\psi^{iv}(\bar r) \right)\,,$$
$$A_{p,n}=2\delta^2(p-1)+\delta((3-2n)p+2n-5)+2n-3$$
$$=\frac{-(2n-3)^2p^2+6(2n-3)p+4n^2-8n-5}{(p+3)^2}<0 \quad\text{for every} \quad p\geq \frac{2n+1}{2n-3}$$
and $B_{p,n}:=\delta(3-p)+3-2n<0$ for every $p>1$. Note that
$\frac{2n+1}{2n-3}\in (1,\frac{n+2}{n-2})$.\par Summing up, if
$\psi$ satisfies assumptions $(H_1)-(H_3)$, $\psi''(0)=0,$
$\psi'''(r)>0$ and $\left(\frac{\psi'(r)}{\psi'''(r)}\right)'\leq
0$ for every $r>0$, then $G$ satisfies the $\Lambda-$property for
every $\frac{2n+1}{2n-3}\leq p <\frac{n+2}{n-2}$.

\subsection{Proof of Theorem \ref{asym sub}}
The statement of (i) is contained in Lemma \ref{l:asym u'}.

\begin{lem} \label{l:dec}
Let the assumptions of Theorem \ref{asym sub} hold and let $u \not
\in H^1(M)$ be a positive solution to \eqref{cauchy}. There exist
no strictly positive constants $C, \beta$ such that $u(r)\leq
C(\psi(r))^{-\beta}$ for all $r\geq 0$.
\end{lem}

\begin{proof} Suppose by contradiction that there exist $C,\beta>0$ such
that $u(r)\leq C(\psi(r))^{-\beta}$ for all $r>0$. Proceeding
exactly as in the proof of Lemma \ref{no exp dec} we arrive to the
estimate \eqref{eq:dec-M}. Integrating \eqref{cauchy} and
exploiting \eqref{eq:dec-M}, we infer $u'(r)\ge -C(\psi(r))^{1-n}$
and any $r>0$, for some constant $C>0$. This shows that $u'\in
L^2(M)$. Another integration then yields $u(r)\le
C\int_r^{+\infty} (\psi(s))^{1-n} ds$ for any $r>0$ and, in turn,
by $(H_3)$ we obtain $u(r)=O((\psi(r))^{1-n})$ as $r\to +\infty$.
This implies $u\in L^2(M)$. We have shown that $u\in H^1(M)$, a
contradiction. The proof of the lemma is complete.
\end{proof}

\begin{lem} \label{l:ex-lim-P}
Let the assumptions of Theorem \ref{asym sub} hold and let $u \not
\in H^1(M)$ be a positive solution to \eqref{cauchy}. Let $P=P(r)$
be defined as in the proof of Theorem \ref{teo}. Then $P(r)$
admits a limit as $r\rightarrow +\infty$.
\end{lem}

\begin{proof}
From the proof of Theorem $\ref{teo}$ we recall that
$P^\prime(r):=K(r)(u^\prime(r))^2$. Hence, by \eq{lim}
$$\displaystyle{\lim_{r\rightarrow +\infty} K(r)}= \displaystyle{\lim_{r\rightarrow +\infty} (\psi(r))^{n-1}\left[\frac{p+3}2-(n-1)(p+1)\psi^\prime(r)\,\frac{\int_0^r(\psi(s))^{n-1}\,ds}{(\psi(r))^{n}}\,\right]=-\infty}$$
and the statement follows.
\end{proof}

\textit{End of the proof of Theorem \ref{asym sub}.} Thanks to
Lemma \ref{l:ex-lim-P} we may put $\gamma:=\lim_{r\to +\infty}
P(r)$. If $\gamma<0$ then $P$ is obviously eventually negative. In
such a case we may proceed exactly as in the proof of Theorem
\ref{asym} and arrive to the estimates (ii) and (iii) of Theorem
\ref{asym sub}.

Suppose now that $\gamma\ge 0$. Since $P$ is eventually
nonincreasing then $P$ is eventually nonnegative. Therefore there
exists $\overline r>0$ such that
\begin{equation} \label{eq:2nd-ord}
(u'(r))^2+\frac{2}{p+1}\frac{(\psi(r))^{n-1}}{\int_0^r
(\psi(s))^{n-1} \, ds} \, u(r)u'(r)+\frac{2}{p+1}(u(r))^{p+1}\ge 0
\qquad \text{for any } r>\overline r \, .
\end{equation}

Suppose now that $\lim_{r\to +\infty} u(r)=0$. Since $l>0$ up to
enlarging $\overline r$, we have that
$$
\frac{1}{(p+1)^2}\left(\frac{(\psi(r))^{n-1}}{\int_0^r
(\psi(s))^{n-1} \, ds}\right)^2 \,
(u(r))^2-\frac{2}{p+1}(u(r))^{p+1}>0 \qquad \text{for any }
r>\overline r \, .
$$
Solving the second order equation in \eqref{eq:2nd-ord} with
respect to $u'(r)$, we arrive to the following alternatives:
either
$$
u'(r)\le -\frac{1}{p+1} \frac{(\psi(r))^{n-1}}{\int_0^r
(\psi(s))^{n-1} \, ds}\,
u(r)-\left[\frac{1}{(p+1)^2}\left(\frac{(\psi(r))^{n-1}}{\int_0^r
(\psi(s))^{n-1} \, ds}\right)^2 \,
(u(r))^2-\frac{2}{p+1}(u(r))^{p+1}\right]^{\frac 12}
$$
or
$$
u'(r)\ge -\frac{1}{p+1} \frac{(\psi(r))^{n-1}}{\int_0^r
(\psi(s))^{n-1} \, ds}\,
u(r)+\left[\frac{1}{(p+1)^2}\left(\frac{(\psi(r))^{n-1}}{\int_0^r
(\psi(s))^{n-1} \, ds}\right)^2 \,
(u(r))^2-\frac{2}{p+1}(u(r))^{p+1}\right]^{\frac 12} .
$$
The first alternative may be excluded since otherwise by
\eqref{lim} we would have
\begin{equation*}
\frac{u'(r)}{u(r)}\le -\frac{n-1-\eps}{p+1}
\frac{\psi'(r)}{\psi(r)} \qquad \text{for any } r>r_\eps
\end{equation*}
for some $\eps\in (0,n-1)$ and $r_\eps>0$. Integration of this
inequality provides a contradiction with Lemma \ref{l:dec}.

Therefore the second alternative holds true. Proceeding similarly
to the proof of \cite[Theorem 2.3]{BGGV} we then obtain
\begin{equation*}
u'(r)\ge -2(u(r))^p \frac{\int_0^r \psi^{n-1}(s) \,
ds}{\psi^{n-1}(r)} \, .
\end{equation*}
Exploiting \eqref{lim} and \eqref{eq:ex-lim}, this implies
\begin{equation*}
\lim_{r\to +\infty} \frac{u'(r)}{u(r)}=0 \, .
\end{equation*}
In particular using again \eqref{eq:ex-lim}, this gives the
validity of Lemma \ref{l:u'/u}. Now one can follow exactly all the
steps of Theorem \ref{asym} and arrive to the proof of part (iii)
if $\frac{\psi}{\psi'}\not\in L^1(0,\infty)$.

Otherwise we arrive to a contradiction with the fact that $u$
vanishes at infinity. This gives the proof of part (ii). We recall
that the existence of $\lim_{r\to +\infty} u(r)$ and the fact that
it is finite follows from Proposition \ref{global}.

\section{Stability}\label{stabil}

\subsection{Proof of Theorems \ref{stability-1}-\ref{stability-2}}
We start with a simpler characterization of stability for radial
solutions of \eqref{LEF}.

\begin{lem} Let $\psi$ satisfy $(H_1)-(H_3)$ and let $u$ be a radial solution of \eqref{LEF}. Then $u$
is stable if and only if
\begin{equation} \label{eq:stab-rad}
\int_{0}^{+\infty} (\chi'(r))^2\,\psi^{n-1}(r)\, dr-
p\,\int_{0}^{+\infty} |u(r)|^{p-1} \chi^2(r)\,\psi^{n-1}(r)\,
dr\geq 0\,,
\end{equation}
for every radial function $\chi\in C_c^{\infty}(M)$.
\end{lem}

\begin{proof} Clearly stability of any solution $u$ of \eqref{LEF}
is equivalent to
\begin{equation}\label{stabRad}
\int_{{\mathbb S}^{n-1}}\int_{0}^{+\infty} [(\varphi_r(r,\Theta))^2+ |\nabla_{{\mathbb S}^{n-1}} \varphi(r,\Theta)|^2\, \psi^{-2}(r)]\,\psi^{n-1}(r)\, dr\,d\Theta
\end{equation}
$$
- p\int_{{\mathbb S}^{n-1}}\int_{0}^{+\infty} |u(r,\Theta)|^{p-1} \varphi^2(r,\Theta)\,\psi^{n-1}(r)\, dr\,d\Theta\geq 0 \quad \forall \, \varphi\in C_c^{\infty}(M).
$$
In particular if $u$ is radial then \eqref{eq:stab-rad} follows
immediately. On the other hand if assume \eqref{eq:stab-rad} we
obtain
\begin{align} \label{eq:radial-stability}
& \int_{{\mathbb S}^{n-1}}\int_{0}^{+\infty}
[(\varphi_r(r,\Theta))^2+ |\nabla_{{\mathbb S}^{n-1}}
\varphi(r,\Theta)|^2\, \psi^{-2}(r)]\,\psi^{n-1}(r)\, dr\,{\rm
d}\Theta \\
\notag & \qquad \geq \int_{{\mathbb S}^{n-1}}\int_{0}^{+\infty}
(\chi_{\Theta}'(r))^2\,\psi^{n-1}(r)\, dr\,d\Theta\geq
\int_{{\mathbb S}^{n-1}}p\,\int_{0}^{+\infty} |u(r)|^{p-1}
\chi_{\Theta}^2(r)\,\psi^{n-1}(r)\, dr\,d\Theta \\
\notag &\qquad = p\int_{{\mathbb S}^{n-1}}\int_{0}^{+\infty}
|u(r)|^{p-1} \varphi^2(r,\Theta)\,\psi^{n-1}(r)\, dr\,{\rm
d}\Theta \, ,
\end{align}
where we have settled $\chi_{\Theta}(r):=\varphi(r,\Theta)$.
\end{proof}

From the next two lemmas it follows that any solution
\eqref{cauchy} with $\alpha>0$ small enough is stable.

\begin{lem} \label{l:stima-alpha}
 Let $\psi$ satisfy assumptions $(H_1)-(H_3)$ and let $u_\alpha$ be a solution of \eqref{cauchy} with $\alpha>0$.
 Then $|u_\alpha(r)|\le \alpha$ for any $r\in [0,+\infty)$.
 \end{lem}

 \begin{proof} Let $F_\alpha(r)=\frac 12 |u_\alpha'(r)|^2+\frac{1}{p+1}
 |u_\alpha(r)|^{p+1}$ be the Lyapunov function corresponding to
 the solution $u_\alpha$. From \eqref{cauchy} one gets that $F_\alpha$ is nonincreasing in $[0,+\infty)$ and hence
 for any $r>0$
 $$
 \frac 1{p+1} \alpha^{p+1}=F_\alpha(0)\ge F_\alpha(r)\ge \frac{1}{p+1} |u_\alpha(r)|^{p+1} \, .
 $$
 This completes the proof.
 \end{proof}

\begin{lem}\label{stabile}
 Let $\psi$ satisfy assumptions $(H_1)-(H_3)$. Furthermore, let $u_{\alpha}$ be a solution to \eqref{cauchy}
 with $|\alpha|\leq \left( \frac{\lambda_1(M)}{p}\right)^{1/(p-1)}$ . Then, $u_{\alpha}$ is stable.
\end{lem}

\begin{proof}
For simplicity, let $\alpha>0$. By Lemma \ref{l:stima-alpha}
$|u_{\alpha}(r)|\le \alpha$ for every $r\ge 0$. The statement
follows by combining \eq{lambda1} with \eq{stab}.
\end{proof}

Next, under suitable assumptions, we show that stable solutions cannot be sign-changing.
\begin{lem}\label{derivative}
Let $\psi$ satisfy assumptions $(H_1)-(H_2)$. Then, any stable
solution to \eqref{cauchy} has constant sign.
\end{lem}

\begin{proof}
By contradiction, let $u$ be a stable solution to \eqref{cauchy} such that $u(R)=0$ for some $R>0$. Next, we set $v_R(r):=u(r)\chi_{[0,R]}(r)\in H_0^1(B_R),$ where $\chi_{[0,R]}(r)$ denotes the characteristic function of the set $[0,R]$ and $B_R$ is the geodesic ball of center $o$ and radius $R$. Standard density arguments yield that $v_R$ is a valid test function in \eq{stabRad}, namely
\begin{equation}\label{test}
\int_{0}^{+\infty} (v_R'(r))^2\,\psi^{n-1}(r)\, dr\,- p\int_{0}^{+\infty} |u(r)|^{p-1} (v_R(r))^2\,\psi^{n-1}(r)\, dr\,\geq 0 \,.
\end{equation}
On the other hand, multiplying the equation in \eqref{cauchy} by $v_R(r) \psi^{n-1}(r)$ and integrating, we get
$$\int_{0}^{+\infty}  (v_R'(r))^2\,\psi^{n-1}(r)\, dr\,=\int_{0}^{+\infty} |u(r)|^{p-1}u\, v_R(r)\,\psi^{n-1}(r)\, dr\,.$$
Recalling the definition of $v_R$, this yields
$$\int_{0}^{+\infty} (v_R'(r))^2\,\psi^{n-1}(r)\, dr\,- p\int_{0}^{+\infty} |u(r)|^{p-1} (v_R(r))^2\,\psi^{n-1}(r)\, dr\,$$
$$=(1-p)\int_{0}^{R} |u(r)|^{p+1}\,\psi^{n-1}(r)\, dr<0\,.$$
The above inequality contradicts \eq{test} and concludes the proof.
\end{proof}

Next we exploit well-know results for the euclidean case to deduce
the following lemma.

\begin{lem}\label{blow}
Let $\psi$ satisfy assumptions $(H_1)-(H_2)$. Let $n\leq 10$ and
$p>1$ or $n\geq 11$ and $1<p<
p_c(n)=\frac{(n-2)^2-4n+8\sqrt{n-1}}{(n-2)(n-10)}$. Then there
exists $\overline \alpha>0$ such that for any $\alpha>\overline
\alpha$, the solution $u_\alpha$ of \eqref{cauchy} is unstable.
\end{lem}
\begin{proof}
We argue by contradiction. Let $u_{\lambda}$ be a stable solution
to \eqref{cauchy} with $\alpha=\lambda^{2/(p-1)}$. As in
\cite[Lemma 7.1]{BGGV}, we define
$$v_{\lambda}(s)=\lambda^{-2/(p-1)} u_{\lambda}\left(\frac{s}{\lambda}\right)\,.$$
Hence, $v_{\lambda}(0)=1$ and $v_{\lambda}$ satisfies
$$v_{\lambda}''(s)+\frac{n-1}{s}\frac{\psi'(s/\lambda)}{\psi(s/\lambda)}\,\frac{s}{\lambda}\,v_{\lambda}'(s)+|v_{\lambda}(s)|^{p-1}v_{\lambda}(s)=0\,.$$
By $(H_1)$ and Ascoli-Arzel\`a Theorem we have that
$v_{\lambda}\rightarrow  \overline{v}$ in $C^1([0,S])$ as
$\lambda\rightarrow +\infty$, for any $0<S<+\infty$, where
$\overline{v}$ solves the equation
$$\overline{v}''(s)+\frac{n-1}{s}\, \overline{v}'(s)+| \overline{v}(s)|^{p-1}\overline{v}(s)=0\,, \qquad \overline{v}(0)=1\,.$$
On the other hand, by assumption $u_{\lambda}$ is stable and from
\eq{stabRad} we have
\begin{equation*}
\int_{0}^{+\infty} (\chi'(r))^2\,(\psi(r))^{n-1}\, dr- p\,
\lambda^2\,\int_{0}^{+\infty} |v_{\lambda}(\lambda r)|^{p-1}
\chi^2(r)\,(\psi(r))^{n-1}\, dr\geq 0\,,
\end{equation*}
for every radial function $\chi\in C_c^{\infty}(M)$. Next, we set
$\eta_{\lambda}(r):=\eta (r\lambda)\in C_c^{\infty}(M)$, for some
$\eta\in C_c^{\infty}(M)$ radial. Choosing $\eta_{\lambda}$ as
test function in the above inequality and performing the change of
variable $s=\lambda r$, we deduce
$$
\int_{0}^{+\infty}
(\eta'(s))^2\,\left(\psi\left(\frac{s}{\lambda}\right)\right)^{n-1}\,
ds- p\,\int_{0}^{+\infty} |v_{\lambda}(s)|^{p-1}
\eta^2(s)\,\left(\psi\left(\frac{s}{\lambda}\right)\right)^{n-1}\,
ds\geq 0\,,$$ for every radial function $\eta\in C_c^{\infty}(M)$.
Let us fix $S\geq 0$ in such a way that ${\rm supp \, }
\eta\subset B_S$. By Lagramge Theorem, for every $s\in [0,S]$
there exist $0<\xi<\frac{s}{\lambda}$ and
$0<|\sigma|<\frac{|\psi''(\xi)|}{2}\frac{s}{\lambda}$ such that
\begin{equation}\label{asin}
\left(\psi\left(\frac{s}{\lambda}\right)\right)^{n-1}=\left(\frac{s}{\lambda}\right)^{n-1}+g(\xi,\sigma)\left(\frac{s}{\lambda}\right)^{n}\, \quad \text{as }\lambda\rightarrow +\infty\,,
\end{equation}
where $g(\xi,\sigma)=(n-1)\left(1+\sigma\right)^{n-2}\frac{\psi''(\xi)}{2}$. This yields
$$
\int_{0}^{+\infty} (\eta'(s))^2\,s^{n-1}\, ds+\int_{0}^{+\infty}
(\eta'(s))^2\,\frac{g(\xi,\sigma)}{\lambda}\,s^{n}\, ds$$
$$- p\,\int_{0}^{+\infty} |v_{\lambda}(s)|^{p-1} \eta^2(s)\,s^{n-1}\, ds- p\,\int_{0}^{+\infty} |v_{\lambda}(s)|^{p-1} \eta^2(s)\,\,\frac{g(\xi,\sigma)}{\lambda}\,s^{n}\, ds\geq 0 \,.$$
Hence, as $\lambda\rightarrow +\infty$, we conclude that
$$
\int_{0}^{+\infty} (\eta'(s))^2\,s^{n-1}\, ds-
p\,\int_{0}^{+\infty} |\overline{v}(s)|^{p-1} \eta^2(s)\,s^{n-1}\,
ds\geq 0\,,$$ for every radial function $\eta\in C_c^{\infty}(M)$
or, equivalently, for every radial function $\eta\in
C_c^{\infty}(\R^n)$. Namely, $\overline{v}$ is a stable solution
to the euclidean equation. Since, by assumption, $n\leq 10$ and
$p>1$ or $n\geq 11$ and $1<p<
p_c(n)=\frac{(n-2)^2-4n+8\sqrt{n-1}}{(n-2)(n-10)}$, this
contradicts \cite[Theorem 1]{farina}.
\end{proof}

Let us introduce some notations which will be used in the sequel.
For any $r>0$, let us denote by $v_\alpha(r)$ the derivative with
respect to the initial value $\alpha$, i.e.
$v(\alpha,r):=\frac{\partial u}{\partial \alpha} (\alpha,r)$. We
will show in Lemma \ref{l:C^1_CDIV} that the function $v_\alpha$
is well-defined. For any $\alpha>\beta$ let us define
\begin{equation*}
\zeta_{\alpha\beta}:=\sup\{r\in (0,\infty):u_\alpha(s)>u_\beta(s)
\ \text{for any } s\in(0,r)\}\in (0,+\infty] \, .
\end{equation*}
When $\zeta_{\alpha,\beta}<+\infty$ then $\zeta_{\alpha,\beta}$ is
the first zero of $u_\alpha-u_\beta$.

\begin{lem} \label{l:C^1_CDIV}
Let $\psi$ a function satisfying $(H_1)-(H_3)$. Let $a,b,R\in\R$ be such that $b>a>0$, $R>0$ and $u_\alpha(r)>0$ for any $r\in [0,R]$ and $\alpha\in [a,b]$. Then for any $r\in [0,R]$, the map $\alpha\mapsto u(\alpha,r)$ is differentiable in $[a,b]$ and moreover for any $\alpha_0\in [a,b]$
\begin{equation} \label{eq:uniform-conv}
\lim_{\alpha\to \alpha_0} \
\sup_{r\in [0,R]} \left| \frac{\partial u}{\partial \alpha}(\alpha,r)-\frac{\partial u}{\partial \alpha}(\alpha_0,r)\right|=0 \, .
\end{equation}
Furthermore for any $\alpha\in [a,b]$ the function $v_\alpha(r):=\frac{\partial
u}{\partial\alpha}(\alpha,r)$, $r\in [0,R]$, is a radial solution of the equation
\begin{equation*}
-\Delta_g v_\alpha=p|u_\alpha|^{p-1} v_\alpha \qquad \text{in } B_R
\, .
\end{equation*}
\end{lem}

\begin{proof} For any $r\in [0,R]$ and $\alpha\in [a,b]$ let us define
\begin{equation*}
w(r)=\frac{u_\alpha(r)-u_{\alpha_0}(r)}{\alpha-\alpha_0}-v_{\alpha_0}(r) \qquad \text{and} \qquad z(r)=w'(r)
\end{equation*}
where by $v_{\alpha_0}$ we mean the unique solution of the Cauchy problem
\begin{equation} \label{cauchy-v}
\begin{cases}
{\ds v''(r)+(n-1)\frac{\psi'(r)}{\psi(r)}\, v'(r)=-p|u_{\alpha}(r)|^{p-1}v(r) }\\
v(0)=1\, \qquad v'(0)=0
\end{cases}
\end{equation}
corresponding to $\alpha=\alpha_0$.
With this notation the following identity holds
\begin{equation*}
z'(r)+(n-1)\frac{\psi'(r)}{\psi(r)}\, z(r)=-\left(\frac{|u_\alpha(r)|^{p-1}u_\alpha(r)-|u_{\alpha_0}(r)|^{p-1}u_{\alpha_0}(r)}{\alpha-\alpha_0}
-p|u_{\alpha_0}(r)|^{p-1}v_{\alpha_0}(r)\right) \, .
\end{equation*}
By elementary estimates and continuous dependence with respect to $\alpha$, we deduce that there exist $\overline \delta>0$ and $C>0$ such that for any $\delta\in (0,\overline \delta)$, $\alpha\in (\alpha_0-\delta,\alpha_0+\delta)\cap [a,b]$ and $r\in [0,R]$
\begin{align} \label{eq:sopra}
& \left|\frac{|u_\alpha(r)|^{p-1}u_\alpha(r)-|u_{\alpha_0}(r)|^{p-1}u_{\alpha_0}(r)}{\alpha-\alpha_0}
-p|u_{\alpha_0}(r)|^{p-1}v_{\alpha_0}(r)\right| \\
\notag & \qquad  \le
\left|\frac{|u_\alpha(r)|^{p-1}u_\alpha(r)-|u_{\alpha_0}(r)|^{p-1}u_{\alpha_0}(r)}{\alpha-\alpha_0}
-p|u_{\alpha_0}(r)|^{p-1} \frac{u_\alpha(r)-u_{\alpha_0}(r)}{\alpha-\alpha_0}   \right|\\
\notag & \qquad \qquad  +
\left|p|u_{\alpha_0}(r)|^{p-1} \frac{u_\alpha(r)-u_{\alpha_0}(r)}{\alpha-\alpha_0}
-p|u_{\alpha_0}(r)|^{p-1}v_{\alpha_0}(r)\right| \\
\notag & \qquad \le C\frac{(u_\alpha(r)-u_{\alpha_0}(r))^2}{\alpha-\alpha_0}+p|u_{\alpha_0}(r)|^{p-1}|w(r)| \, .
\end{align}
By continuous dependence, for any $\eps>0$ there exists $\delta\in(0,\overline\delta)$ such that for any $\alpha\in (\alpha_0-\delta,\alpha_0+\delta)\cap [a,b]$ and $r\in [0,R]$ we have ${\ds \sup_{r\in [0,R]} |u_\alpha(r)-u_{\alpha_0}(r)|<\eps}$ and hence by \eqref{eq:sopra} and the fact that $u_{\alpha_0}\le \alpha_0$ and $v_{\alpha_0}\le 1$, we also obtain
\begin{align*}
& \left|\frac{|u_\alpha(r)|^{p-1}u_\alpha(r)-|u_{\alpha_0}(r)|^{p-1}u_{\alpha_0}(r)}{\alpha-\alpha_0}
-p|u_{\alpha_0}(r)|^{p-1}v_{\alpha_0}(r)\right| \\
& \qquad \le (p\alpha_0^{p-1}+C\eps)|w(r)|+C\eps  \qquad \text{for any } r\in[0,R] \ \text{and} \ \alpha\in (\alpha_0-\delta,\alpha_0+\delta)\cap [a,b] \, .
\end{align*}
Since $w(0)=0$, by the previous inequality we also have
\begin{align*}
& \left|\frac{|u_\alpha(r)|^{p-1}u_\alpha(r)-|u_{\alpha_0}(r)|^{p-1}u_{\alpha_0}(r)}{\alpha-\alpha_0}
-p|u_{\alpha_0}(r)|^{p-1}v_{\alpha_0}(r)\right| \\
& \qquad \le (p\alpha_0^{p-1}+C\eps)\int_0^r |z(s)|ds+C\eps  \qquad \text{for any } r\in[0,R] \ \text{and} \ \alpha\in (\alpha_0-\delta,\alpha_0+\delta)\cap [a,b] \, .
\end{align*}
Simple estimates then yield
\begin{equation*}
|z(r)|\le K(p\alpha_0^{p-1}+C\eps)\int_0^r |z(s)|ds+KC\eps
\end{equation*}
for any $r\in [0,R]$ and $\alpha\in
(\alpha_0-\delta,\alpha_0+\delta)\cap [a,b]$ where ${\ds
K:=\sup_{r\in (0,R]} \frac{\int_0^r
\psi^{n-1}(s)ds}{\psi^{n-1}(r)}}$. Standard Gronwall-type
estimates then yield ${\ds \lim_{\alpha\to \alpha_0} \sup_{r\in
[0,R]} |z(r)|=0}$ and, in turn,
$$
\lim_{\alpha\to \alpha_0} \sup_{r\in [0,R]} |w(r)|=0   \, .
$$
This proves the differentiability with respect to $\alpha$ of the
map $\alpha\mapsto u(\alpha,r)$ and shows that the derivative with
respect to $\alpha$ is a solution of \eqref{cauchy-v}. The proof
of \eqref{eq:uniform-conv} is a consequence of a standard
continuous dependence result for the Cauchy problem
\eqref{cauchy-v}.
\end{proof}

\begin{lem} \label{l:intersection}
Let $\psi$ satisfy $(H_1)-(H_3)$. Let
$\alpha_1>\alpha_2\ge\alpha_3>\alpha_4\ge 0$ be such that
$u_{\alpha_1}(r)>0,u_{\alpha_2}(r)>0,u_{\alpha_3}(r)>0,u_{\alpha_4}(r)\ge
0$ for any $r\in[0,R_0)$ for some $0<R_0\le +\infty$. If
$\zeta_{\alpha_3\alpha_4}\le R_0$ is the first zero of
$u_{\alpha_3}-u_{\alpha_4}$ then $\zeta_{\alpha_1\alpha_2}$, the
first zero  of $u_{\alpha_1}-u_{\alpha_2}$, is finite and it
satisfies $\zeta_{\alpha_1\alpha_2}\leq \zeta_{\alpha_3\alpha_4}$.
\end{lem}

\begin{proof} The proof can be obtained proceeding exactly as in the proof of \cite[Lemma 7.3]{BGGV}.
\end{proof}

We now show that $\lambda_1(B_r)$ diverges as $r\to 0^+$.

\begin{lem} \label{l:lambda_1}
Let $\psi$ satisfy $(H_1)-(H_2)$. Then
$$
\lim_{r\to 0^+} \lambda_1(B_r)=+\infty \, .
$$
\end{lem}

\begin{proof} By $(H_1)-(H_2)$, for any $\overline r$ there exist $0<C_1<C_2$ depending on $\overline r$ such that
\begin{equation*}
C_1 r\le \psi(r)\le C_2 r \qquad \text{for any } r\in [0,\overline
r] \, .
\end{equation*}
Fix $\overline r$ and for any $r\in [0,\overline r]$ let us
consider $\varphi\in C^\infty_c(B_r)$ and the quotient
$$
\frac{\int_{B_r} |\nabla_g \varphi|_g^2 dV_g}{\int_{B_r} \varphi^2
dV_g}\ge\frac{\min\{C_1^{n-1},C_1^{n-3}\}}{C_2^{n-1}}
\frac{\int_{B_r^E} |\nabla\widetilde\varphi(x)|^2 dx}
{\int_{B_r^E} \widetilde\varphi^2(x) dx} \ge
\frac{\min\{C_1^{n-1},C_1^{n-3}\}}{C_2^{n-1}} \, \lambda_1(B_r^E)
$$
where $B_r^E\subset \R^n$ denotes the euclidean ball of radius $r$
centered at the origin, $\widetilde\varphi \in C^\infty_c(B_r^E)$
the function defined by $\widetilde\varphi(x)=\varphi(|x|,x/|x|)$
for any $x\in B_r^E$ and $\lambda_1(B_r^E)$ the first eigenvalue
of $-\Delta$ with Dirichlet boundary conditions in the euclidean
ball $B_r^E$. Since the previous inequality holds for any
$\varphi\in C^\infty_c(B_r)$ then
\begin{equation} \label{eq:ine-eig}
\lambda_1(B_r)\ge \frac{\min\{C_1^{n-1},C_1^{n-3}\}}{C_2^{n-1}} \,
\lambda_1(B_r^E) \, .
\end{equation}
It is well known that thanks to a rescaling argument one has
$\lim_{r\to 0^+} \lambda_1(B_r^E)=+\infty$. Therefore passing to
the limit in \eqref{eq:ine-eig} as $r\to 0^+$ we arrive to the
conclusion of the proof.
\end{proof}

\begin{lem} \label{l:unstable}
Let $\psi$ satisfy $(H_1)-(H_3)$.
Let $\alpha>\beta>0$ and let $u_\alpha$, $u_\beta$ be the corresponding solutions
of \eqref{cauchy}. If $u_\beta$ is unstable then $u_\alpha$ is
unstable.
\end{lem}

\begin{proof}
We assume $u_\alpha$ positive otherwise the statement follows by
Lemma \ref{derivative}.\par First suppose that $u_\alpha$ and
$u_\beta$ have no intersection points. If they are both positive
the conclusion is obvious. If $u_\beta$ changes sign, we reach a
contradiction by Lemma \ref{l:intersection} with
$\alpha_1=\alpha$, $\alpha_2=\alpha_3=\beta$ and $\alpha_4=0$.\par
Next we assume that $u_\alpha$ and $u_\beta$ have at least one
intersection point. Let $\zeta_{\alpha\beta}$ be the first zero of
the function $u_\alpha-u_\beta$. By \eqref{cauchy} we deduce that
$u_\alpha'(\zeta_{\alpha\beta})<u_\beta'(\zeta_{\alpha\beta})$ so
that there exists $\delta>0$ such that $u_\alpha(r)<u_\beta(r)$
for any $r\in (\zeta_{\alpha\beta},\zeta_{\alpha\beta}+\delta)$.
By continuous dependence on the initial datum we deduce that there
exists $\overline\alpha\in (\beta,\alpha)$ such that for any
$\gamma\in [\overline\alpha,\alpha]$ we have
$u_\gamma(r)<u_\beta(r)$ for any $r\in
(\zeta_{\alpha\beta}+\delta/2,\zeta_{\alpha\beta}+\delta)$.

By Lemma \ref{l:intersection} we have that $u_\alpha$ and $u_\gamma$ admit at least one intersection point and moreover
\begin{equation}
\zeta_{\alpha\gamma}\le \zeta_{\overline\alpha\beta}<+\infty \qquad \text{for any } \gamma \in (\overline\alpha,\alpha) \, .
\end{equation}
Let us note that as above one can show that for any $\gamma\in (\overline\alpha,\alpha)$, $u_\alpha<u_\gamma$ in a arbitrarily right neighborhood of
$\zeta_{\alpha\gamma}$.

Let $\{\gamma_k\}\subset [\overline\alpha,\alpha)$ be a sequence such that $\gamma_k\uparrow \alpha$.

Then for any $k$ there exists $r_k\in (\zeta_{\alpha\gamma_k},\zeta_{\overline\alpha\beta}+1)$ such that
\begin{equation*}
\frac{u(\alpha,r_k)-u(\gamma_k,r_k)}{\alpha-\gamma_k}<0
\end{equation*}
and by Lagrange Theorem and Lemma \ref{l:C^1_CDIV} we deduce that there exists $\sigma_k\in (\gamma_k,\alpha)$  such that
\begin{equation*}
v_{\sigma_k}(r_k)=v(\sigma_k,r_k)=\frac{\partial u}{\partial \alpha}(\sigma_k,r_k)=\frac{u(\alpha,r_k)-u(\gamma_k,r_k)}{\alpha-\gamma_k}<0 \, .
\end{equation*}
On the other hand for any $k$, $v(\sigma_k,0)=1>0$ so that there exists $\rho_k\in (0,r_k)$ such that
$v(\sigma_k,\rho_k)=0$. This shows that
\begin{equation} \label{eq:k}
\begin{cases}
-\Delta_g v_{\sigma_k}=p|u_{\sigma_k}|^{p-1} v_{\sigma_k} & \qquad \text{in } B_{\rho_k} \\
v_{\sigma_k}=0 & \qquad \text{on } \partial B_{\rho_k} \, .
\end{cases}
\end{equation}
By the definitions of $r_k$ and $\rho_k$ we easily deduce that
\begin{equation*}
\rho_k\le \zeta_{\overline\alpha\beta}+1 \qquad \text{for any }
k\in\N \, .
\end{equation*}
Multiplying both sides of the above equation by $v_{\sigma_k}$ and
integrating by parts we obtain
\begin{equation*}
\int_{B_{\rho_k}} |\nabla_g v_{\sigma_k}|_g^2 \,
dV_g=\int_{B_{\rho_k}} p|u_{\sigma_k}|^{p-1} v_{\sigma_k}^2 dV_g
\, .
\end{equation*}
We want to show that the sequence $\{\rho_k\}$ is also bounded
away from zero. Since $v_{\sigma_k}\in H^1_0(B_{\rho_k})$ and
$\sigma_k<\alpha$, by Lemma \ref{l:stima-alpha} we have
\begin{align*}
0=\int_{B_{\rho_k}} |\nabla_g v_{\sigma_k}|_g^2 \,
dV_g-\int_{B_{\rho_k}} p|u_{\sigma_k}|^{p-1} v_{\sigma_k}^2 dV_g
\ge (\lambda_1(B_{\rho_k})-p\alpha^{p-1}) \int_{B_{\rho_k}}
v_{\sigma_k}^2 dV_g
\end{align*}
and hence $\lambda_1(B_{\rho_k})\le p\alpha^{p-1}$ for any
$k\in\N$. Therefore if we assume by contradiction that
$\ds{\liminf_{k\to +\infty} \rho_k=0}$ then by Lemma
\ref{l:lambda_1} $\ds{\limsup_{k\to +\infty}
\lambda_1(B_{\rho_k})=+\infty}$, a contradiction.

Then we may define $\ds{\rho_\infty=\liminf_{k\to +\infty}
\rho_k\in (0,+\infty)}$ and the sequence $\{w_k\}\subset H^1(M)$
\begin{equation*}
w_k(x):=
\begin{cases}
\frac{\displaystyle{v_{\sigma_k}(x)}}{\ds{\|v_{\sigma_k}\|_{H^1(M)}}} & \qquad \text{if } x\in B_{\rho_k} \\
0 & \qquad \text{if } x\in M\setminus B_{\rho_k} \, .
\end{cases}
\end{equation*}

Then for any $k$, $w_k$ satisfies problem \eq{eq:k} and
\begin{equation*} 
\int_{M} |\nabla_g w_k|_g^2 \, dV_g-\int_{M} p|u_{\sigma_k}|^{p-1} w_k^2\, dV_g=0  \, .
\end{equation*}
Moreover $\{w_k\}$ is bounded in $H^1(M)$ and hence up to a
subsequence we assume that there exists $w\in H^1(M)$ such that
$w_k\rightharpoonup w$ weakly in $H^1(M)$.

Let $\varphi\in C^\infty_c(B_{\rho_\infty})$ such that for any $k$
large enough $\text{supp}\, \varphi\subset B_{\rho_k}$. Then
\begin{equation*}
\int_{B_{\rho_\infty}} \langle\nabla_g w_k, \nabla_g
\varphi\rangle_g \, dV_g=p\int_{B_{\rho_\infty}}
|u_{\sigma_k}|^{p-1} w_k\varphi \, dV_g \, .
\end{equation*}
Passing to the limit as $k\to +\infty$ and taking into account
that by compact embedding $H^1(B_{\rho_\infty})\subset
L^2(B_{\rho_\infty})$ $w_k\to w$ strongly in
$L^2(B_{\rho_\infty})$, and that by continuity from the initial
data, $u_{\sigma_k}\to u_{\alpha}$ uniformly on compact sets,  we
obtain
\begin{equation} \label{eq:id-var}
\int_{B_{\rho_\infty}} \langle\nabla_g w, \nabla_g
\varphi\rangle_g \, dV_g=p\int_{B_{\rho_\infty}}
|u_{\alpha}|^{p-1} w\varphi \, dV_g  \qquad \text{for any }
\varphi\in C^\infty_c(B_{\rho_\infty}) \, .
\end{equation}
By density, the previous identity holds for any $\varphi\in
H^1_0(B_{\rho_\infty})$.

We claim that $w\in H^1_0(B_{\rho_\infty})$. Up to another
subsequence we may assume that $w_k\to w$ almost everywhere in $M$
with respect to the volume measure $V_g$. But up to a subsequence,
$\rho_k\to \rho_\infty$ so that for almost every $P\in M\setminus
\overline{B}_{\rho_\infty}$, $w_k(P)=0$ for any $k$ large enough.
This proves that $w\equiv 0$ almost everywhere in $M\setminus
\overline{B}_{\rho_\infty}$ and since $w\in H^1(M)$ then $w\in
H^1_0(B_{\rho_\infty})$.

Since $\text{supp}\, w_k\subseteq
B_{\zeta_{\overline\alpha\beta}+1}$ for any $k$, by compact
embedding $H^1(B_{\zeta_{\overline\alpha\beta}+1}) \subset
L^2(B_{\zeta_{\overline\alpha\beta}+1})$, we have that $w_k\to w$ in
$L^2(M)$. Together with \eqref{eq:id-var} and the fact that
$\text{supp} \, w_k\subset B_{\rho_k}$, $\text{supp} \, w\subset
B_{\rho_\infty}$, this implies
\begin{align*}
\lim_{k\to +\infty} & \int_M |\nabla_g w_k|_g^2 \, dV_g=
\lim_{k\to +\infty} \int_{B_{\rho_k}} |\nabla_g w_k|_g^2 \, dV_g =
\lim_{k\to+\infty} \int_{B_{\rho_k}} p|u_{\sigma_k}|^{p-1} w_k^2
\, dV_g\\
& =\lim_{k\to+\infty} \int_{B_{\zeta_{\overline\alpha\beta}+1}}
p|u_{\sigma_k}|^{p-1} w_k^2 \, dV_g
=\int_{B_{\zeta_{\overline\alpha\beta}+1}} p|u_{\alpha}|^{p-1} w^2
\, dV_g =\int_{B_{\rho_\infty}} p|u_{\alpha}|^{p-1} w^2 \, dV_g\\
& =\int_{B_{\rho_\infty}} |\nabla_g w|^2 \, dV_g=\int_M |\nabla_g
w|^2 \, dV_g \, .
\end{align*}
The last identity together with the weak convergence yields
$w_k\to w$ strongly in $H^1(M)$. In particular, since
$\|w_k\|_{H^1(M)}=1$ for any $k\in\N$, then $\|w\|_{H^1(M)}=1$ and
hence $w\not\equiv 0$. Summarizing we have found a nontrivial
function $w\in H^1(M)$ satisfying
\begin{equation*}
\int_M |\nabla_g w|^2 \, dV_g-\int_{M} p|u_{\alpha}|^{p-1} w^2 \,
dV_g=0 \, .
\end{equation*}
Suppose now by contradiction that $u_\alpha$ is stable. Then
\begin{equation*}
\int_{M} |\nabla_g \varphi|_g^2 \, dV_g-\int_{M}
p|u_{\alpha}|^{p-1} \varphi^2\, dV_g\ge 0 \qquad \text{for any }
\varphi\in C^\infty_c(M)
\end{equation*}
and by density the previous inequality holds for any $\varphi\in
H^1(M)$.

This means that $w\in H^1(M)$ is a minimizer of
\begin{equation*}
\inf_{v\in H^1(M)\setminus \{0\}} \frac{\int_{M} |\nabla_g v|_g^2
\, dV_g}{\int_{M} p|u_{\alpha}|^{p-1} v^2\, dV_g} \, .
\end{equation*}
In particular $w\in H^1(M)$ is a solution of the equation
\begin{equation*}
-\Delta_g w=p|u_\alpha|^{p-1} w \qquad \text{in } M
\end{equation*}
and by standard regularity theory $w\in C^2(M)$. In particular $w$
is a classical solution of the ordinary differential equation
\begin{equation*}
-w''(r)-(n-1)\frac{\psi'(r)}{\psi(r)} w'(r)=p|u_\alpha(r)|^{p-1}
w(r) \qquad (r>0) \, .
\end{equation*}
But $w(r)=0$ for any $r>\rho_\infty$ and hence by uniqueness of
the Cauchy problem we infer $w(r)=0$ for any $r>0$ so that
$w\equiv 0$ in $M$, a contradiction.
\end{proof}
Next we define
$$
\alpha_0:=\sup\{\alpha\ge 0:u_\beta \text{ is stable for any }
\beta\in(0,\alpha)\} \, .
$$
By Lemma \ref{stabile} we know that $\alpha_0\in (0,+\infty]$ and
by Lemma \ref{l:unstable} we have that $u_\alpha$ is stable for
any $\alpha\in [0,\alpha_0)$ and it is unstable for any
$\alpha>\alpha_0$ whenever $\alpha_0<+\infty$. In the next lemma
we prove that the set
$$
\mathcal S:=\{\alpha\ge 0: u_\alpha \text{ is stable} \}
$$
is a closed interval.

\begin{lem} \label{l:ordering-0}
Let $\psi$ satisfy $(H_1)-(H_3)$. Then the set $\mathcal S$ is a
closed interval.
\end{lem}

\begin{proof} We have just shown above that $\mathcal S$ is an
interval. It remains to show that if $\alpha_0<+\infty$ then
$\alpha_0\in \mathcal S$. We prove that
$[0,+\infty)\setminus\mathcal S$ is open. Let $\alpha\in
[0,+\infty)\setminus\mathcal S$ so that $u_\alpha$ is unstable.
Hence there exists $\varphi\in C^\infty_c(M)$ such that
\begin{equation} \label{eq:contr}
\int_M |\nabla_g \varphi|_g^2 \, dV_g-\int_M p|u_\alpha|^{p-1}
\varphi^2 dV_g<0 \, .
\end{equation}
We claim that there exists $\delta>0$ such that
\begin{equation*} 
\int_M |\nabla_g \varphi|_g^2 \, dV_g-\int_M p|u_\beta|^{p-1}
\varphi^2 dV_g<0
\end{equation*}
for any $\beta\in (\alpha-\delta,\alpha+\delta)$ or in other words
$[0,+\infty)\setminus\mathcal S$ is open.

Suppose by contradiction that there exists a sequence
$\{\alpha_k\}\subset [0,+\infty)$ such that $\alpha_k\to \alpha$
and
\begin{equation} \label{eq:open-k}
\int_M |\nabla_g \varphi|_g^2 \, dV_g-\int_M p|u_{\alpha_k}|^{p-1}
\varphi^2 dV_g\ge 0 \, .
\end{equation}
Since the $\text{supp} \, \varphi$ is compact by continuous
dependence on the initial data we have that $u_{\alpha_k}\to
u_\alpha$ uniformly in any compact set of $M$. Passing to the
limit in \eqref{eq:open-k} as $k\to +\infty$ we obtain
\begin{equation*}
\int_M |\nabla_g \varphi|_g^2 \, dV_g-\int_M p|u_\alpha|^{p-1}
\varphi^2 dV_g\ge 0
\end{equation*}
in contradiction with \eqref{eq:contr}.
\end{proof}

The estimate $\alpha_0\ge \left(p^{-1} \lambda_1(M)
\right)^{1/(p-1)}$ follows immediately from Lemma \ref{stabile}.
It remains to prove that the inequality is strict under some
additional assumptions. First we prove the following

\begin{lem} \label{l:ex-minimizer} Let $\psi$ satisfy $(H_1)-(H_3)$ and \eqref{alternatives}. Then for any
$\alpha>0$
\begin{equation*}
\Lambda_1(M,\alpha):=\inf_{v\in H^1(M)\setminus \{0\}}
\frac{\int_M |\nabla_g v|_g^2 \, dV_g}{\int_M p|u_\alpha|^{p-1}
v^2 \, dV_g}
\end{equation*}
admits a minimizer.
\end{lem}
\begin{proof}
Let $\{v_k\}\subset H^1(M)$ be a minimizing sequence for
$\Lambda_1(M,\alpha)$ such that
\begin{equation*}
\int_M p|u_\alpha|^{p-1} v_k^2 \, dV_g=1 \, .
\end{equation*}

Then $\{v_k\}$ is bounded in $H^1(M)$ and hence up to a
subsequence there exists $v\in H^1(M)$ such that $v_k
\rightharpoonup v$ weakly in $H^1(M)$. By compact embedding
$H^1(B_r)\subset L^2(B_r)$ we have that $v_k\to v$ strongly in
$L^2(B_r)$ for any $r>0$.

By the assumptions of this lemma, combined with Proposition \ref{global}, formula \eq{serrin}, Theorems \ref{asym sub} and \ref{asym}, we have that $u_\alpha(r)\to 0$ as $r\to
+\infty$. Hence for any $\eps>0$ we may choose $R_\eps>0$ such
that $p|u_\alpha(r)|^{p-1}<\eps$ for any $r>R_\eps$. Hence we
obtain
\begin{align*}
& \left| \int_M p|u_\alpha|^{p-1} v_k^2 \, dV_g -\int_M
p|u_\alpha|^{p-1} v^2 \, dV_g \right|  \leq \left|
\int_{B_{R_\eps}} p|u_\alpha|^{p-1} v_k^2 \, dV_g
-\int_{B_{R_\eps}} p|u_\alpha|^{p-1} v^2 \, dV_g \right|\\
& \qquad + \left| \int_{M\setminus B_{R_\eps}} p|u_\alpha|^{p-1}
v_k^2 \, dV_g -\int_{M\setminus B_{R_\eps}} p|u_\alpha|^{p-1} v^2
\, dV_g \right| \\
& \leq \left| \int_{B_{R_\eps}} p|u_\alpha|^{p-1} v_k^2 \, dV_g
-\int_{B_{R_\eps}} p|u_\alpha|^{p-1} v^2 \, dV_g \right|
+\frac{\eps}{\lambda_1(M)} \,
(\|v_k\|^2_{H^1(M)}+\|v\|^2_{H^1(M)}) \, .
\end{align*}
Passing to the limit as $k\to +\infty$ we obtain
\begin{align*}
& \limsup_{k\to +\infty} \left| \int_M p|u_\alpha|^{p-1} v_k^2 \,
dV_g -\int_M p|u_\alpha|^{p-1} v^2 \, dV_g \right| \le
\frac{2\Lambda_1(M,\alpha)}{\lambda_1(M)} \, \eps \qquad \text{for
any } \eps>0 \, .
\end{align*}
Hence,
$$
\lim_{k\to +\infty} \int_M p|u_\alpha|^{p-1} v_k^2 \, dV_g=\int_M
p|u_\alpha|^{p-1} v^2 \, dV_g \, .
$$
This shows that $v\neq 0$ and that, by the lower semicontinuity of
the $H^1(M)$-norm, $v$ is a minimizer for $\Lambda_1(M,\alpha)$.
\end{proof}

\begin{lem} \label{l:estimate-alpha0}
Let $\psi$ satisfy $(H_1)-(H_3)$ and \eqref{alternatives}. Then
$\alpha_0>\left(\frac{\lambda_1(M)}{p}\right)^{\frac{1}{p-1}}$.
\end{lem}

\begin{proof} Define
$\overline\alpha:=\left(\frac{\lambda_1(M)}{p}\right)^{\frac{1}{p-1}}$.
We claim that $\Lambda(M,\overline\alpha)>1$. To see this, by
Lemma \ref{l:ex-minimizer} we introduce a minimizer $w\in H^1(M)$
of $\Lambda(M,\overline\alpha)$. By Poincar\'e inequality and the
fact that, by Lemma \ref{l:stima-alpha}, $u_{\overline\alpha}\le
\overline\alpha$, we have
\begin{align*}
& \Lambda_1(M,\overline\alpha)= \frac{\int_M |\nabla_g w|_g^2 \,
dV_g}{\int_M p|u_{\overline\alpha}|^{p-1} w^2 \, dV_g} \ge
\frac{\int_M |\nabla_g w|_g^2 \, dV_g}{\lambda_1(M) \int_M w^2 \,
dV_g}\ge 1 \, .
\end{align*}
If assume by contradiction that $\Lambda_1(M,\overline\alpha)=1$
then the inequalities above are equalities and hence $w\in H^1(M)$
is a minimizer for $\lambda_1(M)$. Hence, it solves the equation
$$
-\Delta_g w=\lambda_1(M)w \qquad \text{in } M
$$
and this contradicts the fact that $w$ solves
$$
-\Delta_g
w=\Lambda_1(M,\overline\alpha)|u_{\overline\alpha}|^{p-1} w \qquad
\text{in } M \, .
$$
This completes the proof of the claim. Let us consider a sequence
$\{\alpha_k\}$ such that $\alpha_k \downarrow \overline\alpha$. We
prove that for any large $k$, $\Lambda_1(M,\alpha_k)>1$.

If we proceed by contradiction, we may assume that
$\Lambda_1(M,\alpha_k)\le 1$ for any large $k$.

Let $\{w_k\}\subset H^1(M)$ be a sequence of minimizers for
$\Lambda_1(M,\alpha_k)$ such that $\int_M p |u_{\alpha_k}|^{p-1}
w_k^2 \, dV_g=1$.

Then $\{w_k\}$ is bounded in $H^1(M)$ and up to a subsequence we
may assume that there exists $\overline w\in H^1(M)$ such that
$w_k\rightharpoonup \overline w$ weakly in $H^1(M)$.

For any $\alpha>0$, consider the Lyapunov function
$$
F_\alpha(r):=\frac 12 |u_\alpha'(r)|^2+\frac 1{p+1}
|u_\alpha(r)|^{p+1} \qquad \text{for any } r>0 \, .
$$
For any $\eps>0$ let $R_\eps>0$ be such that
$$
F_{\overline\alpha}(R_\eps)<\eps \, .
$$
We recall that as in the proof of Lemma \ref{l:ex-minimizer} we
have $\displaystyle{\lim_{r\to +\infty} u_{\alpha}(r)=\lim_{r\to +\infty}
u'_{\alpha}(r)=0}$, for any $\alpha>0$.

Since $u_{\alpha_k}(r)\to u_{\overline\alpha}(r)$ and
$u'_{\alpha_k}(r)\to u'_{\overline\alpha}(r)$ for any $r>0$, there
exists $\overline k$ such that
$$
F_{\alpha_k}(R_\eps)<\eps \qquad \text{for any } k>\overline k \,
.
$$
But we know that for any $\alpha>0$ the function $F_\alpha$ is
nonincreasing and hence
$$
F_{\alpha_k}(r)<\eps \qquad \text{for any } r\ge R_\eps \, ,
\text{for any } k>\overline k \, ,
$$
so that
\begin{equation*}
p|u_{\alpha_k}(r)|^{p-1}\le p [(p+1)\eps]^{\frac{p-1}{p+1}} \qquad
\text{for any } r\ge R_\eps \, , \text{for any } k>\overline k \,
.
\end{equation*}
Therefore
\begin{align*}
& \left| \int_M p|u_{\alpha_k}|^{p-1} w_k^2 \, dV_g -\int_M
p|u_{\overline\alpha}|^{p-1} \overline w^2 \, dV_g \right| \le
\left| \int_{B_{R_\eps}} p|u_{\alpha_k}|^{p-1} w_k^2 \, dV_g
-\int_{B_{R_\eps}} p|u_{\overline\alpha}|^{p-1}  w_k^2 \, dV_g
\right|
\\
& \qquad +\left| \int_{B_{R_\eps}} p|u_{\overline\alpha}|^{p-1}
w_k^2 \, dV_g -\int_{B_{R_\eps}} p|u_{\overline\alpha}|^{p-1}
\overline w^2
\, dV_g \right|\\
& \qquad +\left| \int_{M\setminus B_{R_\eps}}
p|u_{\alpha_k}|^{p-1} w_k^2 \, dV_g -\int_{M\setminus B_{R_\eps}}
p|u_{\overline\alpha}|^{p-1} \overline w^2 \, dV_g \right| \\
& \le \sup_{B_{R_\eps}}
|p|u_{\alpha_k}|^{p-1}-p|u_{\overline\alpha}|^{p-1}|
\int_{B_{R_\eps}} w_k^2 \, dV_g+\left| \int_{B_{R_\eps}}
p|u_{\overline\alpha}|^{p-1} w_k^2 \, dV_g -\int_{B_{R_\eps}}
p|u_{\overline\alpha}|^{p-1} \overline w^2 \, dV_g \right|\\
& \qquad + p [(p+1)\eps]^{\frac{p-1}{p+1}} \int_{M\setminus
B_{R_\eps}} w_k^2 \, dV_g+p [(p+1)\eps]^{\frac{p-1}{p+1}}
\int_{M\setminus B_{R_\eps}} \overline w^2 \, dV_g \, .
\end{align*}
By strong convergence $w_k\to \overline w$ in $L^2(B_{R_\eps})$,
uniform convergence $u_{\alpha_k}\to u_\alpha$ in $B_{R_\eps}$,
Poincar\'e inequality, weak lower semicontinuity of the
$H^1(M)$-norm and the fact that $\Lambda(M,\alpha_k)\leq 1$, we
obtain
\begin{align*}
\limsup_{k\to +\infty} \left| \int_M p|u_{\alpha_k}|^{p-1} w_k^2
\, dV_g -\int_M p|u_{\overline\alpha}|^{p-1} \overline w^2 \, dV_g
\right| \le \frac{2}{\lambda_1(M)} \, p
[(p+1)\eps]^{\frac{p-1}{p+1}} \qquad \text{for any } \eps>0 \, .
\end{align*}
This proves that
\begin{equation*}
\lim_{k\to +\infty}  \int_M p|u_{\alpha_k}|^{p-1} w_k^2 \, dV_g
=\int_M p|u_{\overline\alpha}|^{p-1} \overline w^2 \, dV_g \, .
\end{equation*}
Therefore using again the weak lower semicontinuity of the
$H^1(M)$-norm, we obtain
\begin{align*}
& 1<\Lambda_1(M,\overline\alpha)\le \frac{\int_M |\nabla_g
\overline w|_g^2 \, dV_g}{\int_M p|u_{\overline\alpha}|^{p-1}
\overline w^2 \, dV_g} \le \liminf_{k\to +\infty} \frac{\int_M
|\nabla_g w_k|_g^2 \, dV_g}{\int_M p|u_{\alpha_k}|^{p-1} w_k^2 \,
dV_g}= \liminf_{k\to +\infty} \Lambda_1(M,\alpha_k) \, ,
\end{align*}
a contradiction. This proves that $\Lambda_1(M,\alpha_k)>1$ for
any large $k$.

In particular for any large $k$ and any $\varphi\in C^\infty_c(M)$
we have
\begin{align*}
& \int_M |\nabla_g \varphi|_g^2 \, dV_g\ge \Lambda_1(M,\alpha_k)
\int_M p|u_{\alpha_k}|^{p-1} \varphi^2 \, dV_g \ge \int_M
p|u_{\alpha_k}|^{p-1} \varphi^2 \, dV_g \, .
\end{align*}
We found a sequence of values $\alpha_k>\overline\alpha$ such that
$u_{\alpha_k}$ is stable. This completes the proof of the lemma.
\end{proof}

textit{End of the proof of Theorems
\ref{stability-1}-\ref{stability-2}.} The proof of Theorem
\ref{stability-1} simply follows by combining Lemma
\ref{l:ordering-0} with Lemma \ref{stabile} and Lemma \ref{blow}.
The estimate from below on $\alpha_0$ follows from Lemma
\ref{l:estimate-alpha0}.

\subsection{Proof of Theorem \ref{t:ordering}}
Let $\alpha,\beta\in \mathcal S$ with $\alpha>\beta$. We want to
prove that $u_\alpha(r)>u_\beta(r)>0$ for any $r>0$. Suppose by
contradiction that there exists $\overline r>0$ such that
$u_\alpha(\overline r)<u_\beta(\overline r)$. By Lagrange Theorem
and Lemma \ref{l:C^1_CDIV} we deduce that there exists $\sigma\in
(\beta,\alpha)$ such that
\begin{equation*}
v_\sigma(\overline r)=v(\sigma,\overline r)=\frac{\partial
u}{\partial \alpha}(\sigma,\overline r)=\frac{u(\alpha,\overline
r)-u(\beta,\overline r)}{\alpha-\beta}<0
\end{equation*}
and proceeding as in the proof of Lemma \ref{l:unstable} we find
$\rho\in (0,\overline r)$ such that
\begin{equation*}
\begin{cases}
-\Delta_g v_\sigma=p|u_\sigma|^{p-1} v_\sigma & \qquad \text{in }
B_\rho \\
v_\sigma=0 & \qquad \text{on } \partial B_\rho \, .
\end{cases}
\end{equation*}
Testing the above problem with $v_\sigma\in H^1_0(B_\rho)$, we
obtain
\begin{equation*}
\int_{B_\rho} |\nabla_g v_\sigma|_g^2 \, dV_g-\int_{B_\rho}
p|u_\sigma|^{p-1} v_\sigma^2 \, dV_g=0 \, .
\end{equation*}
Next we define $w_\sigma\in H^1(M)$ as the trivial extension of
$v_\sigma$ outside $B_\rho$ in such a way that
\begin{equation*}
\int_{M} |\nabla_g w_\sigma|_g^2 \, dV_g-\int_{M}
p|u_\sigma|^{p-1} w_\sigma^2 \, dV_g=0 \, .
\end{equation*}
But $\sigma\in [0,\alpha_0]$ and hence by Lemma \ref{l:ordering-0}
$u_\sigma$ is stable. Therefore $w_\sigma$ is a minimizer of the
problem
\begin{equation*}
\inf_{v\in H^1(M)\setminus \{0\}} \frac{\int_{M} |\nabla_g v|_g^2
\, dV_g}{\int_{M} p|u_{\sigma}|^{p-1} v^2\, dV_g}
\end{equation*}
and proceeding as in the proof of Lemma \ref{l:unstable} we arrive
to a contradiction.

\subsection{Proof of Theorem \ref{outside-compact-set}}
By Lemma \ref{eq:lambda_1>0} we have $\lambda_1(M)>0$. By Proposition \ref{global}, \eq{serrin}, Theorem
\ref{asym sub} and Theorem \ref{asym} we get the existence of
$R>0$ such that $p|u(r)|^{p-1} \leq \lambda_1(M)$ for every $r>R$.
Let now $B_R$ be the geodesic ball of radius $R$ centered at $o$.
From what just remarked and \eq{lambda1}, inequality \eq{stabRad}
holds for every $\psi\in C_c^{\infty}(M\setminus K)$ and for every
compact $K$ such that $B_{R}\subset K$. In particular, $u$ is
stable outside a compact set.

\subsection{Proof of Proposition \ref{p:liouville-type}}
Since $u$ is stable, from \eq{stabRad} we have
\begin{equation}\label{chi}
\int_{0}^{+\infty} (\chi'(r))^2\,\psi^{n-1}(r) dr-
p\,\int_{0}^{+\infty} |u(r)|^{p-1} \chi^2(r)\,\psi^{n-1}(r)\,
dr\geq 0\,,
\end{equation}
for every radial function $\chi\in C_c^{\infty}(M)$.\par
Inequality \eq{chi} holds for every $\chi\in H^1\cap
L^{\infty}(M)$ with compact support in $M$. Next, we choose
$\chi(r)=u(r)\eta(r)$ with $\eta \in C^1_c(0,+\infty)$ in \eq{chi}
and we get
$$\int_{0}^{+\infty} (u'(r))^2(\eta(r))^2\,\psi^{n-1}(r)\, dr+\int_{0}^{+\infty} (u(r))^2(\eta'(r))^2\,\psi^{n-1}(r)\,dr$$
$$
+\int_{0}^{+\infty} u'(r)u(r)(\eta^2(r))'\,\psi^{n-1}(r)\, dr\geq
p\, \int_{0}^{+\infty} |u(r)|^{p+1} \eta^2(r)\,\psi^{n-1}(r)\, dr
\, .
$$
An integration by parts and \eqref{cauchy} yield
\begin{equation}\label{key ineq}
\int_{0}^{+\infty} (u(r))^2(\eta'(r))^2\,\psi^{n-1}(r)\,dr\geq
(p-1)\,\int_{0}^{+\infty} |u(r)|^{p+1} \eta^2(r)\,\psi^{n-1}(r)\,
dr\,,
\end{equation}
for every radial function $\eta\in C_c^{\infty}(M)$.\par For
$R>0$, let now $\eta_{R}(r)=\eta(r/R)$, where $\eta(r)\in
C^1([0,+\infty))$ is such that $\eta(r)=1$ for $0\leq r<1$ and
$\eta(r)=0$ for $r\geq 2$. Taking $\eta_R$ as test function in
\eq{key ineq}, we get
$$\frac{\|\eta'\|_{L^{\infty}(1,2)}}{R^2}\int_{R}^{2R} (u(r))^2\,\psi^{n-1}(r)\, dr\geq (p-1)\,\int_{0}^{2R} |u(r)|^{p+1} \eta^2(r/R)\,\psi^{n-1}(r)\, dr$$
$$\geq (p-1)\,\int_{0}^{R} |u(r)|^{p+1}\,\psi^{n-1}(r)\, dr\,.$$
As $R\rightarrow +\infty$, recalling that $\int_{0}^{+\infty}
(u(r))^2\,\psi^{n-1}(r)\, dr<+\infty$, we finally conclude that
$$ \int_{0}^{+\infty} |u(r)|^{p+1}\,\psi^{n-1}(r)\, dr\,=0\,.$$
Hence, $u\equiv 0$ in $(0,+\infty)$.

\bibliographystyle{amsplain} 

\begin{thebibliography}{10}
\footnotesize
\setlength{\parskip}{-0.0cm}

\bibitem{BK} C. Bandle, Y. Kabeya,\emph{ On the positive, "radial" solutions of
a semilinear elliptic equation in ${\mathbb H}^n$}, Adv. Nonlinear Anal. 1 (2012), 1-25.

\bibitem{BPT} C. Bandle M.A. Pozio, A. Tesei,\emph{ The Fujita exponent for the Cauchy problem in the
hyperbolic space}, J. Diff. Eq. 251 (2011), 2143-2163.

\bibitem{BS} M. Bhakta, K. Sandeep, \emph{Poincar\'e-Sobolev equations in the hyperbolic space},
Calc. Var. Partial Differential Equations 44 (2012), 247-269.

\bibitem{Besse} A.L. Besse, \emph{Einstein manifolds}, Reprint of the 1987 edition.
Classics in Mathematics. Springer-Verlag, Berlin, 2008.

\bibitem{BGGV} M. Bonforte, F. Gazzola, G. Grillo, J.L. Vazquez,
\emph{Classification of radial solutions to the Emden-Fowler
equation on the hyperbolic space}, to appear in Calc. Var.
Partial Differ. Eq.

\bibitem{bv} H. Brezis, J. L. Vazquez, \emph{Blow-up solutions of
some nonlinear elliptic problems}, Rev. Mat. Univ. Complut. Madrid
10 (1997), 443-469.

\bibitem{Brooks} R. Brooks, \emph{A relation between growth and the spectrum of the Laplacian},
Math. Z. 178 (1981), 501-508.

\bibitem{cc} X. Cabr\'e, A. Capella, \emph{On the stability of radial solutions of semilinear elliptic
equations in all of $\R^n$}, C. R. Acad. Paris, Ser I 338, (2004),
769-774.

\bibitem{cfms} D. Castorina, I. Fabbri, G. Mancini, K. Sandeep, \emph{Hardy-Sobolev extremals, hyperbolic symmetry and
scalar curvature equations}, J. Diff. Eq. 246, (2009), 1187-1206.

\bibitem{dafa} E. N. Dancer, A. Farina, \emph{On the classification of solutions of $-\Delta u=e^u$ on $\R^N$:
stability outside a compact set and applications}, Proc. Amer. Math. Soc. 137 (2009), 1333-1338.

\bibitem{dufa} L. Dupaigne, A. Farina, \emph{Stable solutions of $-\Delta u=f(u)$ in $\R^N$},
J. Eur. Math. Soc. (JEMS) 12 (2010), 855-882.

\bibitem{farina} A. Farina, \emph{On the classification of solutions of the Lane-Emden equation on unbounded domains
of $\R^N$}, J. Math. Pures et Appl. 87, (2007), 537-561.

\bibitem{farina2} A. Farina, \emph{Stable solutions of $-\Delta u=e^u$ on $\R^N$}, C. R. Math.
Acad. Sci. Paris 345, (2007), 63-66.

\bibitem{FMV} A. Farina, L. Mari, E. Valdinoci, \emph{Splitting theorems, symmetry results and overdetermined problems for Riemannian manifolds}, preprint       arXiv:1210.5720 (2012)

\bibitem{fsv} A. Farina, Y. Sire, E. Valdinoci, \emph{Stable solutions of elliptic equations on riemannian
manifolds}, to appear in Journal of Geometric Analysis.

\bibitem{GS} D. Ganguly, K.Sandeep, \emph{Sign changing solutions of the Brezis-Nirenberg problem in the hyperbolic space}, preprint arXiv:1209.5674, 2012.

\bibitem{GW} R.E. Green, H. Wu, \emph{Function theory on manifolds which possess a pole},
\rm LNM 699, Springer, Berlin, 1979.

\bibitem{Gr} A. Grigor'yan, \emph{Isoperimetric inequalities and capacities on Riemannian manifolds.
The Maz'ya anniversary collection}, Vol. 1 (Rostock, 1998),
139-153, Oper. Theory Adv. Appl., 109, Birkh\"auser, Basel, 1999.

\bibitem{hebey} E. Hebey, \emph{Nonlinear analysis on manifolds: Sobolev spaces and
inequalities},
Courant Lecture Notes in Mathematics, vol. 5, New York University Courant Institute of Mathematical Sciences,
New York, 1999.

\bibitem{jl} D. D. Joseph, T. S. Lundgren, \emph{Quasilinear Dirichlet problems
driven by positive sources}, Arch. Rational Mech. Anal. 49
(1972/73), 241-269.

\bibitem{kwong} M. K. Kwong, Y. Li, \emph{Uniqueness of Radial Solutions of Semilinear Elliptic Equations},
Tran. Amer. Math. Soc. 333, (1992), 339-363.

\bibitem{mancini} G. Mancini, K. Sandeep, \emph{On a semilinear equation in $\mathbb{H}^n$},
Ann. Scuola Norm. Sup. Pisa Cl. Sci. 5, vol. VII, (2008), 635-671.

\bibitem{kufner} B. Opic, A. Kufner, \emph{Hardy-type Inequalities}, Pitman Research Notes in Mathematics Series,
Longman, 1990.

\bibitem{pucci} P. Pucci, J. Serrin, \emph{Continuation and limit properties for solutions of strongly nonlinear
second order differential equations}, Asymptotic Analysis 4,
(1991), 97-160 .

\bibitem{Sto} J.J. Stoker, \emph{Differential geometry}, Wiley, New York, 1989.

\end{thebibliography}
\end{document}